\documentclass[a4paper,11pt]{article}
\usepackage[text={170mm,254mm}]{geometry}
\usepackage[T1]{fontenc}
\usepackage[utf8]{inputenc}
\usepackage{lmodern}
\usepackage{authblk}
\usepackage{graphicx}
\usepackage[dvipsnames]{xcolor}
\usepackage{pdflscape}
\usepackage{amsthm}
\usepackage{mathtools}
\usepackage{amssymb}
\usepackage{enumitem}
\usepackage{multirow}
\usepackage{subcaption}
\usepackage[hyperfootnotes=true,colorlinks=true]{hyperref}
\hypersetup{colorlinks=true,linkcolor=blue,citecolor=magenta,urlcolor=black}
\usepackage{tikz}
\usetikzlibrary{calc}
\usetikzlibrary{backgrounds}
\usetikzlibrary{intersections}

\def\x{\mathbf{x}}
\def\y{\mathbf{y}}
\def\z{\mathbf{z}}
\def\s{\mathbf{s}}
\def\w{\mathbf{w}}

\DeclareMathOperator{\Aut}{Aut}
\DeclareMathOperator{\cart}{\Box}
\DeclareMathOperator{\im}{Im}

\theoremstyle{plain}
\newtheorem{theorem}{Theorem}
\newtheorem{lemma}[theorem]{Lemma}
\newtheorem{corollary}[theorem]{Corollary}
\newtheorem{proposition}[theorem]{Proposition}
\newtheorem{conjecture}[theorem]{Conjecture}

\theoremstyle{definition}
\newtheorem{definition}{Definition}

\title{Bipartite edge frustration of nut graphs: proofs and constructions}

\author[1,2,$\ast$]{Nino Ba{\v s}i{\'c}}
\author[3]{Patrick~W.~Fowler}
 
\affil[1]{Faculty of Mathematics and Physics, University of Ljubljana,
  Ljubljana, Slovenia}
\affil[2]{Institute of Mathematics, Physics and Mechanics, Ljubljana, Slovenia}
\affil[3]{Chemistry, School of Mathematical and Physical Sciences, 
University of Sheffield, \linebreak Sheffield S3 7HF, UK}

\begin{document}

\maketitle

\begin{NoHyper}
\renewcommand{\thefootnote}{$\ast$}
\footnotetext{Corresponding author: \texttt{nino.basic@fmf.uni-lj.si}}
\end{NoHyper}

\begin{abstract}
The bipartite edge frustration of a graph $G$, $\varphi(G)$, is the minimum
number of edges that must be removed from $G$ to obtain a graph that is
bipartite. Core graphs are singular graphs on $2$ or more vertices, for which
the kernel of the adjacency matrix contains a full eigenvector. A \emph{nut
graph} is a core graph of nullity one. Nut graphs $G$ are non-bipartite. It is
proved here that $\varphi(G) \geq 2$ if $G$ is a nut graph, i.e.\ every nut
graph requires deletion of at least two edges to produce a bipartite graph. In
addition, we give a constructive proof that for every integer $k \geq 2$ there
exist infinitely many chemical nut graphs $G$ with $\varphi(G) = k$.
Furthermore, for $3 \leq d \leq 8$, the values $k$, for which a $d$-regular nut
graph $G$ with $\varphi(G) = k$ exists, are characterised. It is shown here that
for every integer $k \geq 2$ there exist infinitely many cubic polyhedral nut
graphs $G$ with $\varphi(G) = k$. Two ways of producing larger nut graphs from
smaller are described: the \emph{cut} and \emph{frustration} constructions. One
straightforward consequence is that total double subdivision of a nut graph
(i.e.\ simultaneous insertion of two degree-$2$ vertices into every edge) yields
a nut graph.

\vspace{\baselineskip}

\vspace{\baselineskip}
\noindent \textbf{Math.\ Subj.\ Class.\ (2020):}
05C50, 
05C75, 
05C92  
\end{abstract}

\section{Introduction}

Core graphs are singular graphs on $2$ or more vertices, for which the kernel of
the adjacency matrix contains an eigenvector without zero entries (i.e.\ a full
vector). Nut graphs are defined to have $2$ or more vertices, nullity one, and
the property that they have no zero entry in their non-trivial kernel
eigenvector~\cite{Gutman1998}. In other words, nut graphs are precisely the core
graphs with nullity one. They are of interest for several applications in
chemistry and
physics~\cite{PWF_JCP_2014,Fowler2013,MaxinePaper1,nuttybook,feuerpaper,ScFo2008}.
The graph $K_1$ is bipartite, but excluded from the definition of core and nut
graphs~\cite{Gutman1998}. Core graphs may be bipartite or not, but nut graphs
are non-bipartite~\cite{Gutman1998}. This prompts the question: if nut graphs
are not bipartite, how `far from bipartite' are they?

A natural approach to quantifying this notion of bipartivity is based on
counting those edges that break the defining rule of bipartite graphs,
i.e.\ edges with both endpoints in the same class of a bipartition of $V(G)$. An
edge $e \in E(G)$ is a \emph{frustrated edge} with respect to a given
bipartition $V(G) = V_1 \sqcup V_2$ if both endpoints of $e$ are in the same
class of the bipartition. The \emph{bipartite edge frustration of a graph} $G$,
denoted $\varphi(G)$, is the minimum number of frustrated edges over all
possible bipartitions of the vertex set $V(G)$~\cite{Doslic2007}. In other
words, $\varphi(G)$ is the smallest number of edges that must be deleted from
$G$ in order to obtain a bipartite spanning subgraph. In these terms, the
\emph{bipartivity} of a graph $G$ with $m$ edges is the fraction
$1-\varphi(G)/m$~\cite{Holme2003}. Note that $\varphi(G)$ is a graph invariant,
but that minimum edge count is typically achieved with more than one selection
of frustrated edges. Figure~\ref{fig:smallNuts} shows examples of nut graphs
with small values of the parameter $\varphi(G)$. We will present stratified
tabulations of $\varphi$ for various classes of nut graph.

A large literature in mathematics and computer science deals with versions of
this
problem~\cite{Alon1996,Alon2003,Alon2005,Bondy1974,Edwards1973,Poljak1994,Poljak1995,Shearer1992},
as bipartite edge frustration has been studied under different names. Some
authors denote it by $\tau_\text{odd}(G)$ and call it the smallest size of an
odd-cycle edge transversal~\cite{Kral2004}. Many authors have studied the
maximum cut (max-cut) problem. For example, in~\cite{Alon2003} the notation
$f(G)$ is used to denote the maximum number of edges of $G$ crossing between the
two parts of a bipartition, taken over all possible bipartitions of the vertex
set $V(G)$. In this form, the problem attracted the attention of Erd{\H
o}s~\cite{Erdos1965}. Here we adopt the notation $\kappa(G)$ instead of $f(G)$,
and note that $\kappa(G) + \varphi(G) = |E(G)|$.

It is known that determining bipartite edge frustration of a graph,
$\varphi(G)$, is NP-hard in general~\cite{Yannakakis1978,Yannakakis1981}. The
problem remains NP-hard even within the class of cubic graphs, which implies
that determining the bipartite edge frustration of chemical graphs is also
NP-hard. However, bipartite edge frustration of planar graphs can be computed in
polynomial time~\cite{Aoshima1977,Hadlock1975}.

Much work has been done on establishing the lower bound for the cardinality of
maximum cut sets in certain classes of graphs (see
{\S}\ref{subsec:upperBounds}), which is equivalent to finding the upper bound on
bipartite edge frustration $\varphi(G)$. Here, we establish a sharp lower bound
for $\varphi(G)$ in the case that $G$ is a nut graph.

\begin{figure}[!htbp]
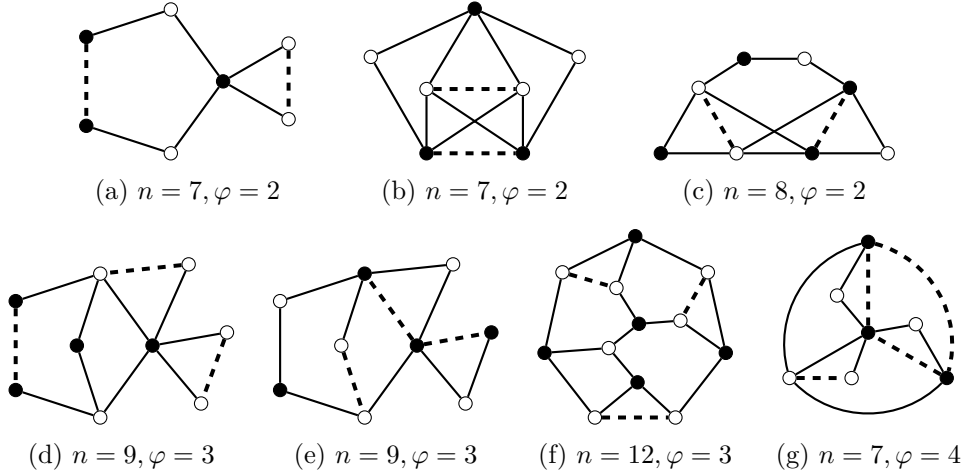

\centering
\subcaptionbox{$n = 7, \varphi = 2$}{\input{img/fig1a.tikz}}
\qquad
\subcaptionbox{$n = 7, \varphi = 2$}{\input{img/fig1b.tikz}}
\qquad
\subcaptionbox{$n = 8, \varphi = 2$}{\input{img/fig1c.tikz}}

\vspace{0.5\baselineskip}
\subcaptionbox{$n = 9, \varphi = 3$}{\input{img/fig1d.tikz}}
\quad
\subcaptionbox{$n = 9, \varphi = 3$}{\input{img/fig1e.tikz}}
\quad
\subcaptionbox{$n = 12, \varphi = 3$}{\input{img/fig1f.tikz}}
\quad
\subcaptionbox{$n = 7, \varphi = 4$} {\input{img/fig1g.tikz}}
\caption{Examples of nut graphs with small bipartite edge frustration. Edges
defining a minimum set of frustrated edges are drawn dashed; vertices of the
graph are coloured to exhibit the implied bipartite spanning graph. Panels (a)
to (c) show the smallest nut graphs with $\varphi = 2$; (d) and (e) show two
choices of frustrated edges for one of the $80$ nut graphs with $\varphi = 3$ on
$n = 9$; (f) shows the Frucht graph, one of the five chemical nut graphs with
$\varphi = 3$ on $n = 12$; (g) shows the smallest nut graph with $\varphi = 4$.
For these small examples, the graph invariant $\varphi$ is easily checked by
hand.}
\label{fig:smallNuts}
\end{figure}

Our main theoretical result, expressed in terms of $\varphi$, is embodied in the
theorem below:

\begin{theorem}
\label{thm:mainResult}
Let $G$ be a core graph. Then $\varphi(G) = 0$ or $\varphi(G) \geq 2$.
\end{theorem}

Recall that every nut graph is a core graph and that all nut graphs are
non-bipartite, i.e.\ $\varphi(G) \neq 0$ if $G$ is a nut graph.
Theorem~\ref{thm:mainResult} has an immediate consequence.

\begin{corollary}
\label{thm:mainCorollary}
Let $G$ be a nut graph. Then $\varphi(G) \geq 2$.
\end{corollary}

Other results presented here focus on the construction and characterisation of
nut graphs with prescribed bipartite edge frustration within various significant
graph classes. The paper is structured as follows. After some background
in~{\S}\ref{sec:background}, a proof of Theorem~\ref{thm:mainResult} is given
in~{\S}\ref{sec:bigtheorem}. In~{\S}\ref{frust}, we give stratified counts for
general, chemical and regular nut graphs by bipartite edge frustration (see also
the~\hyperref[appendix:a]{Appendix}). Then~{\S}\ref{sec:5} gives theoretical
results showing that a nut graph $G$ with $\varphi(G) = k$ can be constructed
for every $k \geq 2$ in the general and chemical graph classes. For $3 \leq d
\leq 8$, we characterise the values $k$, for which a $d$-regular nut graph $G$
with $\varphi(G) = k$ exists. In the same section, we also show that for every
integer $k \geq 2$ there exist infinitely many cubic polyhedral nut graphs $G$
with $\varphi(G) = k$. This section also describes the effects of some
literature constructions on bipartite edge frustration. The fact that all nut
graphs contain at least two frustrated edges motivates definition of novel
constructions (see~{\S}\ref{sec:newConstruct}) for nut graphs, each based on
inflation of a set of edges in the parent nut graph. These are the \emph{cut}
and the \emph{frustration} constructions. The second implies that simultaneous
double subdivision of all edges of a nut graph gives a larger nut graph.
{\S}\ref{Zukunft} offers conclusions, suggestions for further work, and
comparisons with other bipartivity measures.

\section{Preliminaries}
\label{sec:background}

Here we define some notions from graph theory that we will use and present some
established tools that will prove useful.

\subsection{Notation}

Graphs in this paper are finite and simple, unless stated otherwise.
\emph{Chemical graphs} are simple, connected and subcubic, i.e.\ with maximum
degree $\leq 3$. $G(v)$ denotes the set of neighbours of a vertex $v$ in a graph
$G$. For a graph $G$, the adjacency matrix will be denoted by $\mathbf{A}(G)$.
We use $\ker G$ to denote the kernel of the adjacency matrix of $G$, i.e.\ $\ker
G = \ker \mathbf{A}(G)$. The \emph{nullity} of a graph $G$, $\eta(G)$, is the
dimension of the kernel of its adjacency matrix, i.e.\ $\eta (G) = \dim \ker G$.
We say that a vector $\x$ is \emph{full} if it contains no zero entry. In
constructing kernel eigenvectors, we will often find it useful to express the
condition $\mathbf{A}(G) \x = \mathbf{0}$ for $\x \in \ker G$ in the form of an
explicit \emph{local condition} at a vertex $u \in V(G)$:
\begin{equation}
\sum_{v \in G(u)} \x(v) = 0.
\end{equation}
In {\S}\ref{sec:5dot4}, the degree sequence of a graph $G$ is written as a
multiset $\{ d_1^{n_1}, d_2^{n_2}, \ldots \}$, which means that there are $n_i$
vertices of degree $d_i$ in the graph. In other contexts, we may write the
spectrum of a graph $G$ in the form of a multiset $\{ \lambda_1^{k_1},
\lambda_2^{k_2}, \ldots \}$ which means that the eigenvalue $\lambda_i$ has
multiplicity $k_i$ in $G$.

The smallest order $n$ for which nut graphs exist is $n = 7$. There are
precisely $3$ non-isomorphic nut graphs of order $7$, which we call the Sciriha
graphs and denote by $S_1, S_2$ and $S_3$. These are graphs (a), (b) and (g) in
Figure~\ref{fig:smallNuts}, respectively. A vertex $v \in V(G)$ is called a
\emph{core vertex} if $\x(v) \neq 0$ for some $\x \in \ker G$, and is called
\emph{core-forbidden} otherwise~\cite{Sciriha2007}.

We will need notation for various classes of nut graphs. Let $\Delta(G)$ denote
the maximum vertex degree in $G$. $\mathcal{N}_n$ denotes the set of all nut
graphs of order $n$. $\mathcal{N}_n^d$ denotes the set of all $d$-regular nut
graphs of order $n$, while $\mathcal{N}_n^{\{d\}}$ denotes the set of all nut
graphs of order $n$, such that $\Delta(G) \leq d$. Note that $\mathcal{N}_n^3$
is the set of cubic nut graphs of order $n$, while $\mathcal{N}_n^{\{3\}}$ is
the set of chemical nut graphs of order $n$. Hence, $\mathcal{N}_7 = \{S_1, S_2,
S_3\}$. For the special classes of cubic polyhedra and fullerenes on $n$
vertices we use $\mathcal{N}_n^{\text{cp}}$ and $\mathcal{N}_n^{\text{full}}$,
respectively, with $\mathcal{N}_n^{\text{full}} \subset
\mathcal{N}_n^{\text{cp}} \subset \mathcal{N}_n^3$.

A \emph{rooted graph} is a graph in which one vertex is distinguished. Formally,
a rooted graph is a pair $(G, v)$, where $G$ is the graph and $v \in V(G)$ is
the distinguished vertex which we call the \emph{root vertex}. Let $(G_1, v_1)$
and $(G_2, v_2)$ be rooted graphs. The \emph{coalescence of $(G_1, v_1)$ and
$(G_2, v_2)$}, which we denote here as $(G_1, v_1) \odot (G_2, v_2)$, is the
graph obtained from the disjoint union of $G_1$ and $G_2$ by identifying root
vertices $v_1$ and $v_2$. Sciriha has shown that the coalescence $(G_1, v_1)
\odot (G_2, v_2)$ is a nut graph if $G_1$ and $G_2$ are nut
graphs~\cite{Sciriha2008}.

A \emph{cut} in a graph $G$ is a partition of the vertex set $V(G)$ into two
subsets, and we will allow either subset to be empty. Choosing a subset $C
\subseteq V(G)$ uniquely determines the cut: the two subsets of the partition
are $C$ and $V(G) \setminus C$. By a slight abuse of notation, we will call the
set $C$ a cut, implicitly identifying it with the partition $\{ C, V(G)
\setminus C \}$. The \emph{cut set} of a cut $C$, denoted $\delta(C)$, is the
subset of edges $e = uv$ with $u \in C$ and $v \in V(G) \setminus C$. Note that
$\delta(C) = \delta(V(G) \setminus C)$. To confuse matters, the cut set is
simply called a cut by some authors~\cite{Diestel2025}. In the present work, the
set $\omega(C) \coloneqq E(G) \setminus \delta(C)$ will be called the
\emph{frustration set} of the cut $C$. Some authors study a related notion of an
\emph{odd-cycle edge transversal}, which is a set of edges whose removal results
in a bipartite graph~\cite{Faria2012}. Every frustration set is an odd-cycle
edge transversal, but the converse is not true in general. However, the minimum
frustration sets and minimum odd-cycle edge transversals coincide. Note that
$\varphi(G)$ is, in fact, the cardinality of the smallest frustration set of a
graph $G$. Recall that the cardinality of the largest cut set of a graph $G$ is
denoted here by $\kappa(G)$, and, as already pointed out, $\kappa(G) +
\varphi(G) = |E(G)|$.

Graph $G$ is a \emph{planar graph} if $G$ can be embedded in the plane. A graph
$G$ together with a particular embedding is called a \emph{plane graph}. When
the vertices and edges of a plane graph $G$ are erased from the plane, we obtain
one or more connected components, one of which is unbounded. Each such connected
component is called a \emph{face} of $G$. The planar dual of $G$, denoted by
$G^*$, is formed by placing a vertex inside each face of $G$, and connecting
those vertices of $G^*$ whose corresponding faces in $G$ share an edge. Note
that each edge $e$ of $G$ has a corresponding edge $e^*$ in $G^*$; the edge
$e^*$ is called the \emph{dual edge} of $e$. The planar dual $G^*$ need not be a
simple graph, but if $G$ is $3$-connected then $G^*$ is a simple graph. For more
detail, see~\cite{Mohar2001}.

Occasionally, we will refer to the point group symmetry of polyhedral graphs,
when we will use the standard Sch{\"o}nflies group labels (see, for example,
\cite{Atkins1970}).

\subsection{Bipartite edge frustration of planar graphs}

It is known from the classic work of Hadlock on the max-cut problem that
$\kappa(G)$, and therefore $\varphi(G)$, can be computed in polynomial time if
$G$ is a planar graph. Theorems 1--3 and Corollary 1 from \cite{Hadlock1975},
here reformulated in the frustration language, are combined in the following
lemma.

\begin{lemma}[\cite{Hadlock1975}]
\label{lem:hadlock}
Let $G$ be a plane graph and $G^*$ its corresponding planar dual. Let
$\mathcal{P}^* = \{ P^*_1, \ldots, P^*_k \}$ be a collection of non-trivial
edge-disjoint paths in $G^*$ such that each endvertex of each path has odd
degree in $G^*$ and each odd degree vertex of $G^*$ is the endpoint of precisely
one of the paths from $\mathcal{P}^*$. If $\mathcal{P}^*$ is such that
$\sum_{i=1}^k |E(P^*_i)|$ is minimum then
\begin{equation}
\varphi(G) = \sum_{i=1}^k |E(P^*_i)|
\end{equation}
and the set of frustrated edges is $\{ e \in E(G) \mid e^* \in \bigcup_{i=1}^k
E(P^*_i) \}$.
\end{lemma}

\noindent Lemma~\ref{lem:hadlock} is the basis for a polynomial-time algorithm.
Moreover, we will also use Lemma~\ref{lem:hadlock} in the proofs of
Lemma~\ref{lem:antiprismFrustration} in {\S}\ref{sec:5dot4} and
Theorem~\ref{thm:polyhedral} in {\S}\ref{subsec:fowler}.

\subsection{Established bounds for
\texorpdfstring{$\boldsymbol{\varphi(G)}$}{φ(G)}}
\label{subsec:upperBounds}

Study of the bipartite edge frustration of nut graphs can benefit from the
established results. Results on the lower bound for $\kappa(G)$ are easily
re-expressed as upper bounds on bipartite edge frustration $\varphi(G)$; see
also~\cite{Poljak1995}. In these terms, the early result of Erd{\H o}s from
1965~\cite{Erdos1965} translates to $\varphi(G) \leq \lceil \frac{m}{2} -
\frac{n}{4} \rceil$, where $G$ is a graph of order $n$ and size $m$ without
isolated vertices. Later investigations focused on triangle-free
graphs~\cite{Bondy1986,Erdos1988,Erdos1992,Hopkins1982,Locke1982,Staton1980}.
For example, the result of Thomassen~\cite[Corollary 2]{Thomassen2006}
translates to $\varphi(G) \leq \lfloor \frac{7n}{24} + \frac{7}{6} \rfloor$,
where $G$ is a planar, cubic, $3$-connected, triangle-free graph of order $n$
(i.e., a triangle-free, cubic, polyhedral graph). Likewise, the improved result
by Cui and Wang~\cite[Theorem 1.2]{Cui2009} translates to $\varphi(G) \leq
\lfloor \frac{9n}{32} + \frac{9}{16} \rfloor$ for the same class of graphs.

Bipartite edge frustration has been studied extensively for fullerenes. It is
known that $\varphi(G) \geq 6$ if $G$ is a fullerene graph and that $\varphi(G)
\geq 12$ if $G$ is an IPR fullerene~\cite{Doslic2007}. An IPR (Isolated Pentagon
Rule) fullerene is one in which no two pentagonal faces are incident.
In~\cite{Faria2012}, a conjecture from~\cite{Doslic2007} was proved: $\varphi(G)
\leq \sqrt{\frac{12}{5}n}$ if $G$ is a fullerene of order $n$ and that the
equality holds if and only if $G$ is a fullerene of order $60k^2$ for some
integer $k$ with $\Aut(G) \cong I_\mathrm{h}$ (which implies that $G$ is, in
fact, an IPR fullerene). Nicodemos and Stehl{\'{\i}}k~\cite{Nicodemos2018}
extended this result to $\varphi(G) \leq \sqrt{\frac{(p + 3t)n}{5}}$, where $G$
is a cubic polyhedral (i.e.\ planar and $3$-connected) graph of order $n$ with
all faces of size at most $6$, that has $p$ pentagonal and $t$ triangular faces.

\section{Proof of Theorem~\ref{thm:mainResult}}
\label{sec:bigtheorem}

Two preliminary lemmas are required for the proof. With the notation $G + u_1
u_2$ we denote the graph that is obtained from $G$ by adding the edge $u_1 u_2$.
The necessary lemmas are:

\begin{lemma}
\label{lemma:3}
Let $G$ be a bipartite graph with bipartition $V(G) = U \sqcup W$, $|U| \geq 2$.
Let $u_1, u_2 \in U$ and let $\x \in \ker (G + u_1u_2)$. Then $\x(u_1) = 0$ or
$\x(u_2) = 0$.
\end{lemma}

\begin{proof}
The adjacency matrix of $G$ can be written in block form
\begin{equation}
\mathbf{A}(G) = \begin{bmatrix}
\mathbf{0} & \mathbf{B} \\
\mathbf{B}^\intercal & \mathbf{0}
\end{bmatrix}.
\end{equation}
Suppose that vertices $u_1$ and $u_2$ correspond to the first two rows (columns)
of the adjacency matrix. The adjacency matrix of $G + u_1u_2$ is then
\begin{equation}
\mathbf{A}(G+ u_1u_2) = \begin{bmatrix}
\mathbf{P} & \mathbf{B} \\
\mathbf{B}^\intercal & \mathbf{0}
\end{bmatrix},
\end{equation}
where $P_{i,j} = 1$ if $(i, j) \in \{ (1, 2), (2, 1)\}$, and $P_{i,j} = 0$
otherwise.

Suppose that $\w \in \ker (G+ u_1u_2)$. The vector $\w$ can be written in block
form as $\w = \begin{bmatrix} \x \\ \y \end{bmatrix}$.
The equation $\mathbf{A}(G+ u_1u_2) \w = \mathbf{0}$ implies
\begin{align}
\mathbf{P} \x + \mathbf{B} \y & = \mathbf{0}, \label{eq:line1} \\
\mathbf{B}^\intercal \x & =  \mathbf{0}, \label{eq:line2}
\end{align}
and from \eqref{eq:line1} we have
\begin{equation}
\x^\intercal \mathbf{P} \x + \x^\intercal \mathbf{B} \y = 0.
\label{eq:theMainEq}
\end{equation}
Note that~\eqref{eq:line2} implies $\x^\intercal \mathbf{B} =
(\mathbf{B}^\intercal \x)^\intercal = \mathbf{0}$, and thus $\x^\intercal
\mathbf{B} \y = 0$. Equation~\eqref{eq:theMainEq} reduces to $\x^\intercal
\mathbf{P} \x = 0$. The matrix $\mathbf{P}$ has only two non-zero entries, and
this reduces to: $2\x(u_1)\x(u_2) = 0$. The result follows.
\end{proof}

\begin{lemma}
\label{lemma:4}
Let $G$ be a bipartite graph with bipartition $V(G) = U \sqcup W$, $|U| \geq 2$.
Let $u_1, u_2 \in U$ and let $G' \coloneqq G + u_1u_2$. Then there exists a
vertex $u_0 \in V(G')$ such that $\x(u_0) = 0$ for all $\x \in \ker G'$.
\end{lemma}

\begin{proof}
If $\eta(G') = 0$ there is nothing to prove. Suppose that $\eta(G') > 0$ and let
$\x \in \ker G'$ be a non-trivial vector such that $\x(u_1) \neq 0$ or $\x(u_2)
\neq 0$. If no such vector exists, then all kernel eigenvectors vanish on both
$u_1$ and $u_2$; taking $u_0$ to be either finishes the proof. By
Lemma~\ref{lemma:3}, $\x(u_1) = 0$ or $\x(u_2) = 0$. Without loss of generality,
suppose that $\x(u_1) = 0$. But then $\x(u_2) \neq 0$. Now take $u_0 = u_1$. To
obtain a contradiction, suppose there exists $\y \in \ker G'$, such that
$\y(u_0) = \y(u_1) \neq 0$. Lemma~\ref{lemma:3} gives us $\y(u_1) \y(u_2) = 0$,
therefore $\y(u_2) = 0$. Now consider the vector $\x + \y$. It is clear that $\x
+ \y \in \ker G'$. But vector $\x + \y$ does not satisfy Lemma~\ref{lemma:3}, a
contradiction.
\end{proof}

\noindent The two lemmas enable a concise proof of Theorem~\ref{thm:mainResult},
and hence Corollary~\ref{thm:mainCorollary}.

\begin{proof}[Proof of Theorem~\ref{thm:mainResult}]
Let $G$ be a core graph. Suppose that $\varphi(G) = 1$. Therefore, $G$ is
obtained from a bipartite graph $H$ by adding an edge that connects vertices
from the same set of the bipartition. By Lemma~\ref{lemma:4}, there exists a
vertex $u_0 \in V(G)$ such that $\x(u_0) = 0$ for all $\x \in \ker G$. This
contradicts the fact that $G$ is a core graph. Therefore, $\varphi(G) \neq 1$.
Furthermore, if $G$ is a nut graph, $\varphi(G) \neq 1$ implies $\varphi(G) \geq
2$.
\end{proof}

\section{Bipartite edge frustration of small nut graphs}
\label{frust}

It has been known since their first appearance in the
literature~\cite{Gutman1998} that nut graphs are not bipartite. In other words,
$\varphi(G) \geq 1$ if $G$ is a nut graph. Corollary~\ref{thm:mainCorollary}
strengthens this result to: $\varphi(G) \geq 2$ if $G$ is a nut graph.

In the associated GitHub repository~\cite{GitHubSupplement}, we provide an
implementation of an algorithm for finding bipartite edge frustration of general
graphs, and a polynomial-time algorithm for planar (and plane) graphs; these
codes were used to produce the tables for nut graphs in the present section.
Counts of nut graphs for orders $n \leq 13$ (i.e., the values $|\mathcal{N}_n|$)
are given in~\cite{CoolFowlGoed-2017} and collections of nut graphs up to order
$12$ are available from the House of Graphs~\cite{HoG1,HoG2}. In
Table~\ref{tbl:tbl1} we stratify the counts of nut graphs by their bipartite
edge frustration. Note that nut graphs do not exist for orders $n < 7$. The
table shows that the bound in Corollary~\ref{thm:mainCorollary} is best
possible, and inspection of the data suggests that for every integer $k \geq 2$
there exists a nut graph $G$ with $\varphi(G) = k$, a supposition that is proved
in \S\ref{sec:5} (see Proposition~\ref{prop:allExist}).

\begin{table}[!ht]
\centering
\begin{tabular}{|rr||rrrrrr|}
\hline
\multicolumn{2}{|c||}{$n$} & 7 & 8 & 9 & 10 & 11 & 12 \\
\hline
\multicolumn{2}{|c||}{$|\mathcal{N}_n|$} & 3 & 13 & 560 & $12\,551$ & $2\,060\,490$ & $208\,147\,869$ \\
\hline
\hline
\multirow[t]{23}{*}{$\varphi$} & 2 & 2 & 1 & 41 & 20 & 2\,406 & 1\,685 \\
& 3 &  &  & 80 & 90 & 14\,218 & 22\,530 \\
& 4 & 1 & 9 & 180 & 1\,051 & 83\,924 & 365\,516 \\
& 5 &  &  & 122 & 2\,104 & 211\,743 & 1\,988\,453 \\
& 6 &  & 3 & 96 & 3\,788 & 406\,562 & 7\,980\,839 \\
\cline{1-8}
& 7 &  &  & 25 & 2\,683 & 472\,409 & 19\,454\,433 \\
& 8 &  &  & 15 & 1\,938 & 420\,866 & 35\,148\,273 \\
& 9 &  &  & 1 & 590 & 259\,539 & 43\,644\,801 \\
& 10 &  &  &  & 230 & 127\,804 & 42\,410\,292 \\
& 11 &  &  &  & 48 & 43\,870 & 29\,844\,346 \\
\cline{1-8}
& 12 &  &  &  & 8 & 13\,435 & 16\,852\,683 \\
& 13 &  &  &  & 1 & 2\,975 & 7\,138\,318 \\
& 14 &  &  &  &  & 622 & 2\,454\,715 \\
& 15 &  &  &  &  & 100 & 655\,151 \\
& 16 &  &  &  &  & 13 & 154\,506 \\
\cline{1-8}
& 17 &  &  &  &  & 2 & 26\,146 \\
& 18 &  &  &  &  & 1 & 4\,489 \\
& 19  &  &  &  &  & 1 & 576 \\
& 20 &  &  &  &  &  & 105 \\
& 21 &  &  &  &  &  & 7 \\
\cline{1-8}
& 22 &  &  &  &  &  & 2 \\
& 23 &  &  &  &  &  & 2 \\
& 24 &  &  &  &  &  & 1 \\
\hline
\end{tabular}%

\caption{Bipartite edge frustration of small nut graphs.}
\label{tbl:tbl1}
\end{table}

Counts of chemical nut graphs on orders $n \leq 22$ are given
in~\cite{CoolFowlGoed-2017}. In Table~\ref{tbl:2}, we report their
stratification by bipartite edge frustration. An example of a small chemical nut
graph for each bipartite edge frustration $2 \leq \varphi \leq 6$ is shown in
Figure~\ref{fig:smallChemicalNuts}. Note that chemical nut graphs do not exist
for orders $n < 9$ or for order $n = 10$. Table~\ref{tbl:2} hints that perhaps
all bipartite edge frustrations $k \geq 2$ can be realised even within the realm
of chemical nut graphs (and this is indeed true; see
Theorem~\ref{prop:chemicalRealm}).

The enumeration of chemical graphs (not necessarily nut graphs), stratified by
bipartite edge frustration, on orders $n \leq 20$ can be found in the GitHub
repository~\cite{GitHubSupplement}. An upper bound on bipartite edge frustration
for chemical graphs is easy to obtain:

\begin{table}[!htb]
\centering
\begin{tabular}{|r|r||r|r|r|r|r|r|}
\hline
& & \multicolumn{6}{c|}{$\varphi$} \\
\cline{3-8}
$n$ & $|\mathcal{N}_n^{\{3\}}|$ & 2 & 3 & 4 & 5 & 6 & 7 \\
\hline
\hline
9 & 1 & 1 & 0 & & & & \\
11 & 8 & 8 & 0 & & & & \\
12 & 9 & 4 & 5 & 0 & & & \\
13 & 27 & 18 & 9 & 0 & & & \\
14 & 23 & 7 & 9 & 7 & & & \\
15 & 414 & 283 & 113 & 18 & 0 & & \\
16 & 389 & 124 & 213 & 52 & 0 & & \\
17 & 7\,941 & 3\,520 & 3\,529 & 880 & 12 &  & \\
18 & 8\,009 & 1\,369 & 3\,414 & 3\,076 & 147 & 3 &  \\
19 & 67\,970 & 24\,579 & 32\,851 & 10\,294 & 246 & 0 & \\
20 & 51\,837 & 8\,228 & 23\,268 & 18\,367 & 1\,934 & 40 &  \\
21 & 1\,326\,529 & 378\,123 & 645\,929 & 287\,518 & 14\,864 & 95 & 0 \\
22 & 1\,372\,438 & 150\,695 & 562\,064 & 587\,284 & 70\,974 & 1\,421 & 0 \\
\hline
\end{tabular}%

\caption{Bipartite edge frustration of chemical nut graphs. A zero entry in the
table means that though chemical graphs exist for these parameters, none are nut
graphs.}
\label{tbl:2}
\end{table}

\begin{figure}[!ht]
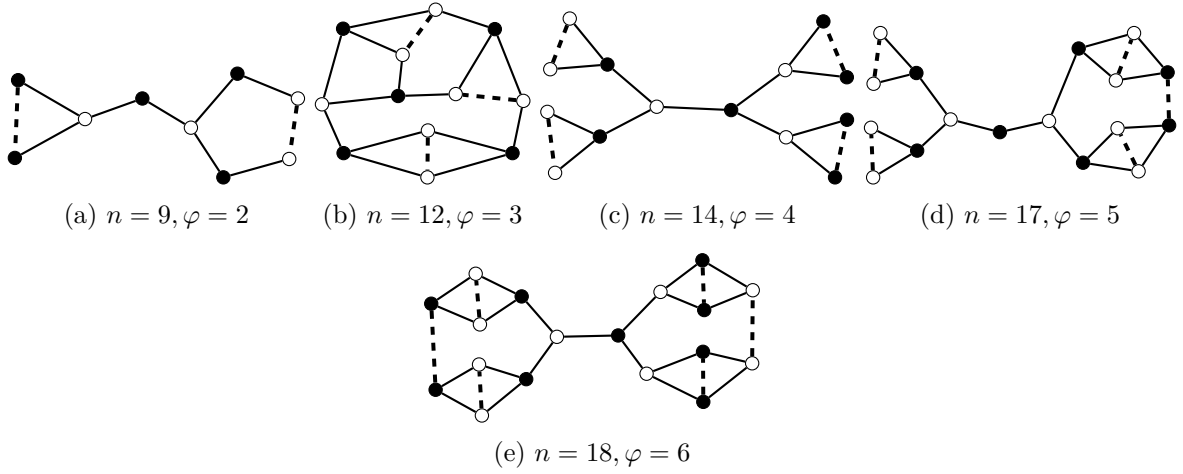

\centering
\subcaptionbox{$n = 9, \varphi = 2$}{\input{img/chemical9_bef2.tikz}}
\subcaptionbox{$n = 12, \varphi = 3$}{\input{img/chemical12_bef3.tikz}}
\vspace{0.5\baselineskip}
\subcaptionbox{$n = 14, \varphi = 4$}{\input{img/chemical14_bef4.tikz}}
\subcaptionbox{$n = 17, \varphi = 5$}{\input{img/chemical17_bef5.tikz}}
\subcaptionbox{$n = 18, \varphi = 6$}{\input{img/chemical18_bef6.tikz}}
\caption{Examples of smallest chemical nut graphs with bipartite edge
frustration $2 \leq \varphi \leq 6$.}
\label{fig:smallChemicalNuts}
\end{figure}

\begin{proposition}
\label{prop:bunting}
Let $G$ be a chemical graph of order $n$, such that $G \not\cong K_4$. Then
$\varphi(G) \leq \lfloor \frac{n}{3} \rfloor$.
\end{proposition}

\begin{proof}
Let $G$ be a chemical graph. If $G$ is bipartite then $\varphi(G) = 0$ and there
is nothing to prove. Assume that $G$ is non-bipartite. By Brooks' Theorem it
follows that either $\chi(G) = 3$, where $\chi(G)$ denotes the chromatic number,
or $G$ is $K_4$, in which case $\chi(G) = 4$. By assumption, $G$ is not $K_4$,
thus $\chi(G) = 3$. Take a proper $3$-colouring of $G$ and divide the vertices
into colour classes $V(G) = V_1 \sqcup V_2 \sqcup V_3$, and, without loss of
generality, we may assume that $|V_3| \leq \frac{n}{3}$. The graph induced on
vertices $V_1 \cup V_2$ is bipartite. We will now recolour vertices of $V_3$,
one by one. Note that all neighbours of a vertex $v \in V_3$ are in $V_1 \cup
V_2$. Let $d_i(v)$, $i \in \{1, 2\}$, be the number of neighbours of $v$ in
$V_i$. Observe that $\min\{ d_1(v), d_2(v) \} \leq 1$. If $d_1(v) \leq 1$, then
recolour the vertex $v$ with colour $1$, otherwise with colour $2$. For each $v
\in V_3$, at most one frustrated edge is obtained. Thus, $\varphi(G) \leq |V_3|
\leq \lfloor \frac{n}{3} \rfloor$.
\end{proof}

Let us define the infinite \emph{bunting graph}, denoted $B_\infty$, as an
infinite graph with vertex set $V(B_\infty) = \mathbb{Z}$ and edge set
$E(B_\infty) = \{ (3k, 3k+1), (3k, 3k+2), (3k, 3k - 2), (3k+1, 3k+2) \mid k \in
\mathbb{Z}\}$. Then the bunting graph of order $n$, denoted $B_n$, is the
subgraph of $B_\infty$ induced on the vertices $\{0, 1, \ldots, n - 1\}$. The
bound in Proposition~\ref{prop:bunting} is sharp, as graph $B_n$ attains the
equality. Furthermore, for every value $0 \leq t \leq \frac{n}{3}$, there exists
a chemical graph $G$ of order $n$ with $\varphi(G) = t$. These graphs can be
obtained from $B_n$ by deleting the appropriate number of frustrated edges. (As
an edge in a minimum frustration set is never a bridge, the graph remains
connected.)

\begin{figure}[!ht]
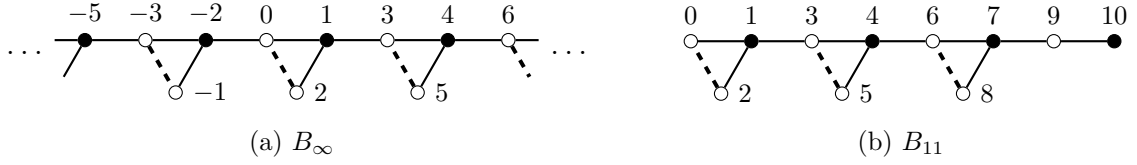

\centering
\subcaptionbox{$B_\infty$}{\input{img/bunting_inf.tikz}}
\qquad
\subcaptionbox{$B_{11}$}{\input{img/bunting_11.tikz}} 
\caption{The infinite bunting graph $B_\infty$ and the finite bunting graph on
$11$ vertices.}
\end{figure}

An interesting subset of chemical graphs consists of the cubic polyhedra (cubic
$3$-connected simple planar graphs), which serve as models for unsaturated
carbon cages. In particular, cubic polyhedra with all faces pentagonal and
hexagonal are models for fullerene molecules~\cite{Atlas2007}. Nut graphs among
cubic polyhedra and fullerene graphs are available from the House of
Graphs~\cite{HoG1,HoG2}. In Table~\ref{tbl:3} and Table~\ref{tbl:4}, we report
stratified counts for cubic polyhedral nut graphs and fullerene nut graphs,
respectively. All our examples of fullerene nut graphs have some adjacent
pentagons, and hence cannot meet the general fullerene upper
bound~\cite{Faria2012}. We note that among nut fullerenes of order $n \leq 250$
we have found only one that has an odd value of $\varphi$, namely the
$42$-vertex fullerene illustrated in Figure~\ref{fig:fullerFru7}(a).
Figure~\ref{fig:fullerFru7}(b) illustrates the application of
Lemma~\ref{lem:hadlock} on this fullerene.

Enumerations of cubic polyhedral graphs (not necessarily nut graphs), stratified
by bipartite edge frustration, on orders $n \leq 26$, and of fullerene graphs on
orders $n \leq 130$ can be found in the GitHub
repository~\cite{GitHubSupplement}.

Table~\ref{tbl:3} suggests that infinitely many polyhedral nut graphs exist for
any given $\varphi \geq 2$. This is proved here as Theorem~\ref{thm:polyhedral}.
The smallest polyhedral nut graph with $\varphi = 3$ is the Frucht graph (see
Figure~\ref{fig:smallNuts}(f)). More examples of small cubic polyhedral nut
graphs (for $\varphi \in \{2, 4, 5, 6, 7, 8, 9\}$) can be found in
Figure~\ref{fig:nutsSmallCubicPolyhedra}.

\begin{figure}[!htbp]
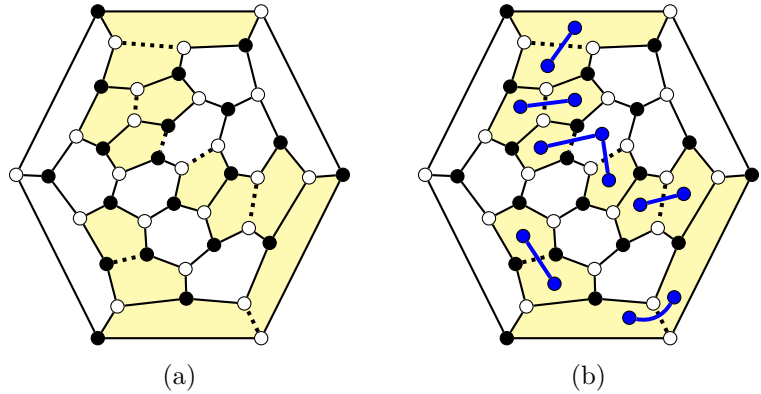

\centering
\subcaptionbox{}{\input{img/fullerene42_bef7.tikz}}
\qquad 
\subcaptionbox{}{\input{img/fullerene42_bef7_lemma3.tikz}}
\caption{The only known nut fullerene with an odd value of $\varphi$. In this
case $\varphi = 7$. In the spiral numbering~\cite{Atlas2007}, this fullerene is
labelled $42{:}12$. It has $C_\text{s}$ symmetry. The right panel illustrates
the application of Lemma~\ref{lem:hadlock}.}
\label{fig:fullerFru7}
\end{figure}

\begin{table}[!ht]
\centering
\begin{tabular}{|r|r||r|r|r|r|r|r|r|r|r|r|}
\hline
\multirow{2}{*}{$n$} & \multirow{2}{*}{$|\mathcal{N}_n^{\textup{cp}}|$} & \multicolumn{8}{c|}{$\varphi$} \\
\cline{3-10}
 &  & 2 & 3 & 4 & 5 & 6 & 7 & 8 & 9 \\
\hline
\hline
12 & 2 & 1 & 1 &  &  &  &  &  & \\
18 & 285 & 24 & 95 & 157 & 9 &  &  &  & \\
24 & 62\,043 & 848 & 5\,665 & 29\,354 & 22\,042 & 4\,114 & 20 &  & \\
26 & 4 &  &  &  & 1 & 2 & 1 &  & \\
28 & 316 &  &  & 6 & 73 & 190 & 47 &  & \\
30 & 16\,892\,864 & 37\,883 & 331\,689 & 3\,211\,580 & 6\,531\,474 & 5\,709\,061 & 1\,037\,292 & 33\,885 & \\
32 & 3\,676 &  &  & 57 & 459 & 1\,305 & 1\,430 & 417 & 8 \\
34 & 447\,790 & 19 & 86 & 3\,384 & 29\,611 & 136\,594 & 205\,136 & 70\,499 & 2\,461 \\
\hline
\end{tabular}%

\caption{Bipartite edge frustration of cubic polyhedral nut graphs.}
\label{tbl:3}
\end{table}

\begin{figure}[!ht]
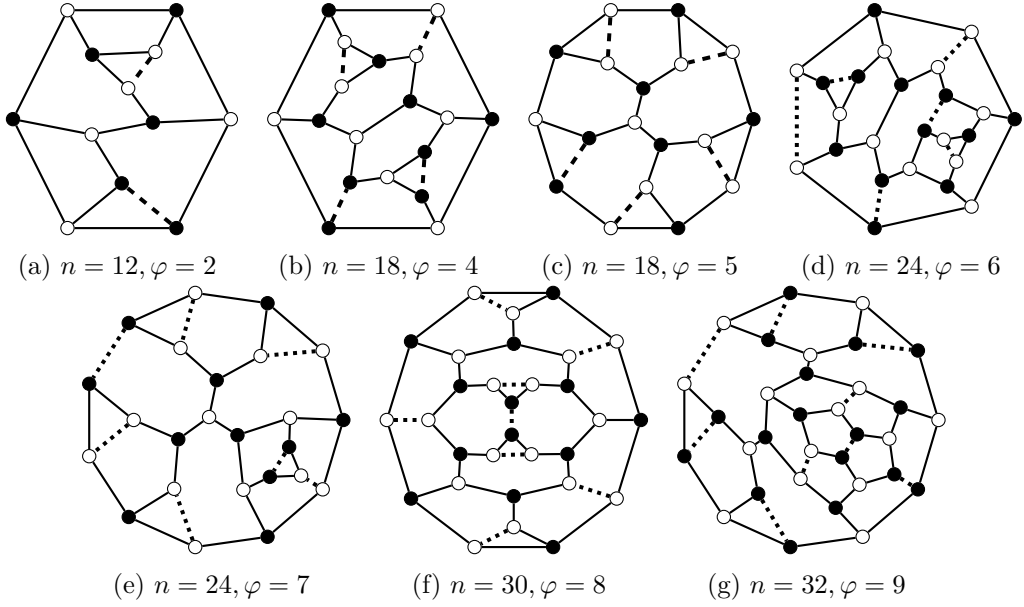

\centering
\subcaptionbox{$n = 12, \varphi = 2$} { 
\input{img/polyhedron12_bef2.tikz}
}
\subcaptionbox{$n = 18, \varphi = 4$} { 
\input{img/polyhedron18_bef4.tikz}
}
\subcaptionbox{$n = 18, \varphi = 5$} { 
\input{img/polyhedron18_bef5.tikz}
}
\subcaptionbox{$n = 24, \varphi = 6$} { 
\input{img/polyhedron24_bef6.tikz}
}

\subcaptionbox{$n = 24, \varphi = 7$} { 
\input{img/polyhedron24_bef7.tikz}
}
\subcaptionbox{$n = 30, \varphi = 8$} { 
\input{img/polyhedron30_bef8.tikz}
}
\subcaptionbox{$n = 32, \varphi = 9$} { 
\input{img/polyhedron32_bef9.tikz}
}
\caption{Examples of smallest cubic polyhedral nut graphs with bipartite edge
frustration $2 \leq \varphi \leq 9$.}
\label{fig:nutsSmallCubicPolyhedra}
\end{figure}

\begin{table}[!ht]
\centering
\begin{tabular}{|r|r||r|r|r|r|r|r|r|}
\hline
& & \multicolumn{7}{c|}{$\varphi$} \\
\cline{3-9}
$n$ & $|\mathcal{N}_n^\mathrm{\,full}|$ & 6 & 7 & 8 & 10 & 12 & 14 & 16 \\
\hline
\hline
 36 & 1 & 1 &  &  &  &  &  &   \\
 42 & 1 &  & 1 &  &  &  &  &   \\
 44 & 1 & 1 &  &  &  &  &  &   \\
 48 & 2 & 2 &  &  &  &  &  &   \\
 52 & 2 &  &  & 2 &  &  &  &   \\
 \cline{1-9}
 60 & 6 & 3 &  & 3 &  &  &  &   \\
 72 & 2 & 1 &  & 1 &  &  &  &   \\
 82 & 1 &  &  &  & 1 &  &  &   \\
 84 & 8 & 3 &  & 3 & 2 &  &  &   \\
 96 & 5 & 1 &  & 1 & 3 &  &  &   \\
 \cline{1-9}
  108 & 7 & 3 &  & 2 & 2 &  &  &   \\
 120 & 5 & 1 &  & 2 & 2 &  &  &   \\
 132 & 14 & 3 &  & 3 & 5 & 3 &  &   \\
 144 & 6 & 1 &  & 2 & 2 & 1 &  &   \\
 156 & 11 & 3 &  & 2 & 3 & 3 &  &   \\
  \cline{1-9}
  160 & 1 &  &  &  & 1 &  &  &   \\
 168 & 11 & 1 &  & 1 & 5 & 4 &  &   \\
 180 & 16 & 3 &  & 3 & 2 & 6 & 2 &   \\
 192 & 8 & 1 &  & 2 & 2 & 2 & 1 &   \\
 204 & 19 & 3 &  & 2 & 6 & 7 & 1 &   \\
  \cline{1-9}
  216 & 9 & 1 &  & 2 & 2 & 2 & 2 &   \\
 228 & 21 & 3 &  & 3 & 2 & 6 & 5 & 2 \\ 
 240 & 16 & 1 &  & 2 & 5 & 4 & 2 & 2 \\
 \hline
\end{tabular}%

\caption{Bipartite edge frustration of fullerene nut graphs.}
\label{tbl:4}
\end{table}

The House of Graphs also provides collections of small regular nut graphs. In
the~\hyperref[appendix:a]{Appendix}, we report stratification tables (similar to
Tables~\ref{tbl:tbl1} and~\ref{tbl:2}) for regular nut graphs. There are gaps in
some columns, i.e., missing combinations of $n$ and $\varphi$ for certain
degrees. The following explains this.

\begin{proposition}
\label{prop:explainMissingFrust}
The following statements hold.
\begin{enumerate}[label=(\alph*)]
\item Let $G$ be an $r$-regular graph where $r \equiv 0 \pmod{4}$. Then
$\varphi(G)$ is even.
\item Let $G$ be an $r$-regular graph where $r \equiv 2 \pmod{4}$. Then
$\varphi(G) \equiv |V(G)| \pmod{2}$. Moreover, if $\varphi(G)$ is odd, then
$\varphi(G) \geq \frac{r}{2}$.
\item Let $G$ be an $r$-regular graph where $r$ and $\varphi(G)$ are both odd.
Then $\varphi(G) \geq r$.
\end{enumerate}
\end{proposition}

\begin{proof}
Colour the vertices of graph $G$ black and white so that there are $\varphi(G) =
\varphi_b + \varphi_w$ frustrated edges, comprising $\varphi_b$ black-black
edges and $\varphi_w$ white-white. Let $n_b$ and $n_w$ be respective numbers of
black and white vertices. Counting the edges connecting black and white vertices
in two ways, we obtain
\begin{equation}
r ( n_b  - n_w )  = 2 (\varphi_b - \varphi_w).
\label{eq:fruParityEq}
\end{equation}
Claims (a) and (b) follow by straightforward parity arguments.

To prove the second assertion in (b), suppose that $\varphi(G)$ is odd.
By~\eqref{eq:fruParityEq}, $\varphi_b - \varphi_w$ is a multiple of
$\frac{r}{2}$. Without loss of generality, we may assume that $\varphi_b \geq
\varphi_w$. As $\varphi_b, \varphi_w \geq 0$, it follows that $\varphi_b =
\varphi_w + k \frac{r}{2}$ for some non-negative integer $k$. Suppose that $k =
0$. Then $\varphi(G) = 2\varphi_b$. But we assumed that $\varphi(G)$ is odd,
therefore $k \geq 1$. This implies $\varphi(G) \geq 2\varphi_w + \frac{r}{2}
\geq \frac{r}{2}$.

We now prove (c). If $r$ is odd then \eqref{eq:fruParityEq} implies $\varphi_b
\equiv \varphi_w \pmod{r}$. As in the previous paragraph, we may assume that
$\varphi_b \geq \varphi_w$, from which we obtain $\varphi_b = \varphi_w + k r$
for $k \geq 1$. This implies $\varphi(G) \geq 2\varphi_w + r \geq r$.
\end{proof}

\section{Nut and core graphs of prescribed bipartite edge frustration}
\label{sec:5}

In this section, we show that every bipartite edge frustration
$\widetilde{\varphi} \geq 2$ is realised by some (chemical) nut graph. We also
show that for every $\widetilde{\eta} \geq 2$ and every $\widetilde{\varphi}
\neq 1$ we can find some core graph $G$ with $\eta(G) = \widetilde{\eta}$ and
$\varphi(G) = \widetilde{\varphi}$. For $3 \leq d \leq 8$, we also characterise
all values $\widetilde{\varphi}$ that are realised by some $d$-regular nut
graph. Next, we show that an established literature construction preserves
bipartite edge frustration. Finally, we show that every $\widetilde{\varphi}
\geq 2$ is realised by some cubic polyhedral nut graph. But first, we show that
bipartite edge frustration is additive with respect to the coalescence
construction.

\subsection{The coalescence construction}

The following lemma, which follows by repeated application of Lemma~0 from
\cite{Doslic2010}, shows that under certain circumstances bipartite edge
frustration is additive. We include a short proof for completeness.

\begin{lemma}
\label{lem:blockAdditive}
Let $B_1, B_2, \ldots, B_k$ be the blocks of a graph $G$. Then $\varphi(G) =
\sum_{i=1}^k \varphi(B_i)$.
\end{lemma}

\begin{proof}
Recall that a block is a maximal subgraph that has no
cut-vertices~\cite{West2001}. A block can be an isolated vertex, a bridge, or a
maximal $2$-connected subgraph. Any given cycle, and thus any odd cycle, of $G$
is entirely contained in a single block. Deleting edges from a block has no
effect on cycles in the other blocks. The claim follows.
\end{proof}

Recall that $(G_1, v_1) \odot (G_2, v_2)$ is a nut graph if $G_1$ and $G_2$ are
nut graphs. Vertices $v_1 \in V(G_1)$ and $v_2 \in V(G_2)$ may be chosen
arbitrarily. If the bipartite edge frustration of $G_1$ and $G_2$ is known, then
the value for the coalesced graph is easily computed. Observe that the vertex
which emerged by identification of $v_1$ and $v_2$ is a cut-vertex in $(G_1,
v_1) \odot (G_2, v_2)$. Lemma~\ref{lem:blockAdditive} thus yields the following.

\begin{corollary}
\label{lem:coalAdditive}
Let $G_1$ and $G_2$ be graphs and let $v_1 \in V(G_1)$ and $v_2 \in V(G_2)$.
Then
\begin{equation}
\varphi((G_1, v_1) \odot (G_2, v_2)) = \varphi(G_1) + \varphi(G_2).
\end{equation}
\end{corollary}

\noindent In~\cite{Sciriha2016}, Ali et al.\ state a graph-theoretic version of
Cauchy's Interlacing Theorem, which we use in the proof of
Lemma~\ref{lem:coreCoalescenceNullity}:

\begin{theorem}[{\cite[Theorem 1.2]{Sciriha2016}}]
\label{thm:etaVertexDeletion}
Let $G$ be a graph of order at least $2$ and let $v \in V(G)$. Then
\begin{equation}
\eta(G) - 1 \leq \eta(G - v) \leq \eta(G) + 1.
\end{equation}
\end{theorem}

We will also use the following:

\begin{proposition}
\label{prop:deleteVertexFromCoreGraph}
Let $G$ be a core graph and let $v \in V(G)$. Then $\eta(G - v) = \eta(G) - 1$.
\end{proposition}

\noindent As every vertex of a core graph is a core vertex, the above
proposition follows immediately from~\cite[Proposition 13]{Sciriha2008}.

\begin{lemma}
\label{lem:coreCoalescenceNullity}
Let $G_1, G_2$ be core graphs and let $v_1 \in V(G_1)$ and $v_2 \in V(G_2)$ be
arbitrary vertices. Then $(G_1, v_1) \odot (G_2, v_2)$ is a core graph and
$\eta((G_1, v_1) \odot (G_2, v_2)) = \eta(G_1) + \eta(G_2) - 1$.
\end{lemma}

\begin{proof}
Let us denote $\eta_1 = \eta(G_1)$, $\eta_2 = \eta(G_2)$, and $H = (G_1, v_1)
\odot (G_2, v_2)$. Let $\{ \x_1, \x_2, \ldots, \x_{\eta_1} \}$ and $\{ \y_1,
\y_2, \ldots, \y_{\eta_2} \}$ be the sets of basis vectors for $\ker G_1$ and
$\ker G_2$, respectively. Since $G_1$ is a core graph, we can assume that
$\x_1(v) \neq 0$ for all $v \in V(G_1)$ and $\x_i(v_1) = 0$ for all $i \geq 2$.
Similarly, $\y_1(v) \neq 0$ for all $v \in V(G_2)$ and $\y_i(v_2) = 0$ for all
$i \geq 2$. Moreover, we may assume that $\x_1(v_1) = \y_1(v_2)$. Let us define
$\w \colon V(H) \to \mathbb{R}$, such that
\[
\w(v) = \begin{cases}
\x_1(v) & \text{if }  v \in V(G_1), \\
\y_1(v) & \text{if }  v \in V(G_2).
\end{cases}
\]
Furthermore, let $\widetilde{\x}_i \colon V(H) \to \mathbb{R}$ for $2 \leq i
\leq \eta_1$ and $\widetilde{\y}_i \colon V(H) \to \mathbb{R}$ for $2 \leq i
\leq \eta_2$ be defined as follows:
\begin{align*}
\widetilde{\x}_i(v) & = \begin{cases}
\x_i(v) & \text{if }  v \in V(G_1), \\
0 & \text{if }  v \in V(G_2);
\end{cases} \\
\widetilde{\y}_i(v) & = \begin{cases}
0 & \text{if }  v \in V(G_1), \\
\y_i(v) & \text{if }  v \in V(G_2).
\end{cases}
\end{align*}
It is easy to see that vectors $\w, \widetilde{\x}_2, \ldots,
\widetilde{\x}_{\eta_1}, \widetilde{\y}_2, \ldots, \widetilde{\y}_{\eta_2}$ all
reside in $\ker H$. Since $\w$ is a full vector, $H$ is a core graph.

To see that $\eta(H) = \eta(G_1) + \eta(G_2) - 1$, we will show that $\{ \w,
\widetilde{\x}_2, \ldots, \widetilde{\x}_{\eta_1}, \widetilde{\y}_2, \ldots,
\widetilde{\y}_{\eta_2} \}$ is a basis of $\ker H$. First, we show that vectors
$\w, \widetilde{\x}_2, \ldots, \widetilde{\x}_{\eta_1}, \widetilde{\y}_2,
\ldots, \widetilde{\y}_{\eta_2}$ are linearly independent. To obtain a
contradiction, suppose that there exist some multipliers $\mu_0, \mu_2^{(1)},
\ldots, \mu_{\eta_1}^{(1)}, \mu_2^{(2)}, \ldots, \mu_{\eta_2}^{(2)}$, not all
zero, so that
\[
\s \coloneqq \mu_0 \w + \mu_2^{(1)} \widetilde{\x}_2 + \cdots + 
\mu_{\eta_1}^{(1)} \widetilde{\x}_{\eta_1} + \mu_2^{(2)} \widetilde{\y}_2 + 
\cdots + \mu_{\eta_2}^{(2)} \widetilde{\y}_{\eta_2} = 0.
\]
Observe that $\s(v_1) = \mu_0 \w (v_1)$. As $\w (v_1) \neq 0$, it follows that
$\mu_0 = 0$. Without loss of generality, assume that a non-zero multiplier is
among $\mu_2^{(1)}, \ldots, \mu_{\eta_1}^{(1)}$. But then $\mu_2^{(1)} \x_2 +
\cdots + \mu_{\eta_1}^{(1)} \x_{\eta_1} = 0$, which contradicts the fact that
$\{ \x_1, \x_2, \ldots, \x_{\eta_1} \}$ is a basis for $\ker G_1$. In other
words, $\eta(H) \geq \eta(G_1) + \eta(G_2) - 1$. Now, we show that $\eta(H) \leq
\eta(G_1) + \eta(G_2) - 1$. By Proposition~\ref{prop:deleteVertexFromCoreGraph},
$\eta(G_1 - v_1) = \eta_1 - 1$ and $\eta(G_2 - v_2) = \eta_2 - 1$. By
Theorem~\ref{thm:etaVertexDeletion},
\begin{equation}
\eta(H - v_1) \geq \eta(H) - 1,
\label{eq:coalCoreEq1}
\end{equation} 
but the graph $H - v_1$ is a disjoint union of graphs $G_1 - v_1$ and $G_2 -
v_2$, and therefore
\begin{equation}
\eta(H - v_1) = \eta(G_1 - v_1) + \eta(G_2 - v_2) =  \eta_1 + \eta_2 - 2.
\label{eq:coalCoreEq2}
\end{equation} 
Together \eqref{eq:coalCoreEq1} and \eqref{eq:coalCoreEq2} yield the claim.
\end{proof}

\noindent Lemma~\ref{lem:coreCoalescenceNullity} has some overlap with
\cite[Proposition 18]{Sciriha2008}. The following generalises both:

\begin{lemma}
Let $G_1, G_2$ be singular graphs and let $v_1 \in V(G_1)$ and $v_2 \in V(G_2)$
be core vertices in $G_1$ and $G_2$, respectively. Then
$(G_1, v_1) \odot (G_2, v_2)$ is singular and \[ 
\eta((G_1, v_1) \odot (G_2, v_2)) = \eta(G_1) + \eta(G_2) - 1.
\]
\end{lemma}

\noindent Proof follows by modification of proofs of either
Lemma~\ref{lem:coreCoalescenceNullity} or Proposition 18 in~\cite{Sciriha2008}.

\subsection{Nut and core graphs with prescribed bipartite edge frustration}

The next proposition shows that the bound in Corollary~\ref{thm:mainCorollary}
is best possible.

\begin{proposition}
\label{prop:allExist}
For every integer $k \geq 2$ there exists a nut graph $G$ such that $\varphi(G)
= k$.
\end{proposition}

\begin{proof}
The proof goes by induction. There exist nut graphs with bipartite edge
frustration $2$ and $3$, e.g., $S_1$ in Figure~\ref{fig:smallNuts}(a) and the
Frucht graph in Figure~\ref{fig:smallNuts}(f), respectively. A graph $G$ with
$\varphi(G) = k$ can be obtained from a graph $G'$ with $\varphi(G') = k - 2$ by
coalescence of $G'$ and $S_1$ (for arbitrary choice of root vertices) by using
Corollary~\ref{lem:coalAdditive}.
\end{proof}

Recall that if a core graph is bipartite then its nullity is at least $2$. The
next proposition shows that bipartite core graphs of arbitrary nullity
$\widetilde{\eta} \geq 2$ are easily obtained.

\begin{proposition}
Let $a, b \geq 2$. $K_{a,b}$ is a core graph and $\eta(K_{a, b}) = a + b - 2$.
\end{proposition}

\begin{proof}
The spectrum of $K_{a,b}$ is $\{ \sqrt{ab}, 0^{a + b - 2}, -\sqrt{ab} \}$, see
e.g.~\cite[p.~8]{haemers}. Suppose that $V(K_{a,b}) = \{ v_1, \ldots, v_a, u_1,
\ldots, u_b \}$, where $v_1, \ldots, v_a$ are in the first part of the
bipartition and $u_1, \ldots, u_b$ are in the second part. It is easy to see
that $\x \colon V(K_{a,b}) \to \mathbb{R}$, defined by $\x(v_1) = 1 - a$,
$\x(u_1) = 1 - b$ and $\x(w) = 1$ if $w \notin \{v_1, u_1\}$, is a full vector
such that $\x \in \ker K_{a,b}$.
\end{proof}

\noindent This proposition has an immediate consequence.

\begin{corollary}
\label{cor:completeBipiEta}
For every $\ell \geq 2$, there exists a core graph $G$ with $\eta(G) = \ell$ and
$\varphi(G) = 0$.
\end{corollary}

\begin{theorem}
For every $\ell \geq 2$ and every $k \geq 2$ or $k = 0$, there exists a core
graph $G$ with $\eta(G) = \ell$ and $\varphi(G) = k$.
\end{theorem}

\begin{proof}
Case $k = 0$ is covered by Corollary~\ref{cor:completeBipiEta}. By
Proposition~\ref{prop:allExist} there exists a nut graph $G_1$ with
$\varphi(G_1) = k$. By Corollary~\ref{cor:completeBipiEta} there exists a core
graph $G_2$ with $\eta(G_2) = \ell$ and $\varphi(G_2) = 0$. Let $v_1 \in V(G_1)$
and $v_2 \in V(G_2)$ be arbitrary vertices. By Corollary~\ref{lem:coalAdditive},
$\varphi((G_1, v_1) \odot (G_2, v_2)) = k$. By
Lemma~\ref{lem:coreCoalescenceNullity}, $(G_1, v_1) \odot (G_2, v_2)$ is a core
graph with $\eta((G_1, v_1) \odot (G_2, v_2)) = \ell$.
\end{proof}

\subsection{Chemical nut graphs with prescribed bipartite edge frustration}

Proposition~\ref{prop:allExist} asserted the existence of a nut graph $G$ with
$\varphi(G) = k$ for all $k \geq 2$. This holds even if we further require $G$
to be a chemical nut graph. To demonstrate this, we introduce an \emph{ad hoc}
construction which we call the \emph{prism splice} construction\footnote{Note
that the term \emph{splice} has sometimes been used~\cite{Doslic2005b,
Yarahmadi2010} for the operation on a pair of graphs that we are calling
\emph{coalescence}. Our definition of splice does not coincide with that used
in~\cite{Doslic2005b, Yarahmadi2010}.}. Let $G_1$ and $G_2$ be disjoint simple
graphs and let $u_1 v_1 \in E(G_1)$ and $u_2 v_2 \in E(G_2)$. The \emph{splice}
of $(G_1, u_1, v_1)$ and $(G_2, u_2, v_2)$ is the graph with vertex set $V(G_1)
\cup V(G_2)$ and edge set $(E(G_1) \cup E(G_2) \cup \{ u_1u_2, v_1v_2 \})
\setminus \{ u_1v_1, u_2v_2\}$. Here, we denote this graph $(G_1, u_1, v_1)
\bowtie (G_2, u_2, v_2)$. In the particular construction that we call the prism
splice, and denote by $\mathrm{PS}(G, e)$, we make the splice of $(G, u, v)$,
where $e = uv$, and $(Y_3, u', v')$ where $Y_3 \cong C_3 \cart K_2$ is the
$3$-prism graph, and $u' v'$ is an edge belonging to a $3$-cycle (see
Figure~\ref{fig:splice}). Note that graphs $(G_1, u_1, v_1) \bowtie (G_2, u_2,
v_2)$ and $(G_1, u_1, v_1) \bowtie (G_2, v_2, u_2)$ are not necessarily
isomorphic. In the case of the prism splice, given the symmetry of the prism, we
can swap vertices $u'$ and $v'$ in $(G, u, v) \bowtie (Y_3, u', v')$, and the
resulting graph is unique.

\begin{figure}[!ht]
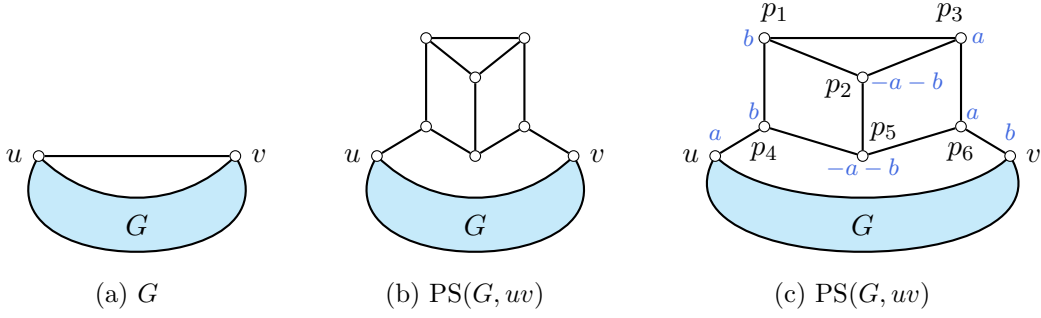

\centering
\subcaptionbox{$G$}{\input{img/splice.tikz}}
\quad
\subcaptionbox{$\mathrm{PS}(G, uv)$}{\input{img/splice2.tikz}}
\quad
\subcaptionbox{$\mathrm{PS}(G, uv)$}{\input{img/splice3.tikz}}
\caption{The prism splice construction showing the parent graph $G$, the spliced
graph $\mathrm{PS}(G, uv)$, and kernel eigenvector $\y$ constructed in
Lemma~\ref{lem:batEigenvectors} (entries in blue).}
\label{fig:splice}
\end{figure}

\begin{lemma}
\label{lem:batEigenvectors}
Let $G$ be a graph with $uv \in E(G)$ and let the extra vertices of
$\mathrm{PS}(G, uv)$ be labelled $p_1, \ldots, p_6$ as shown in
Figure~\ref{fig:splice}(c). If\, $\y \in \ker \mathrm{PS}(G, uv)$ then
\[
\y(p_3) = \y(p_6) = \y(u), \qquad \y(p_1) = \y(p_4) = \y(v)
\qquad\text{and}\qquad \y(p_2) = \y(p_5) = - \y(u) - \y(v).
\]
\end{lemma}

\begin{proof}
Let us write $a = \y(u)$ and $b = \y(v)$. The local conditions at vertices $p_1,
\ldots, p_6$ are
\[
\begin{aligned}
a + \y(p_1) + \y(p_5) & = 0, & 
\y(p_2) + \y(p_3) + \y(p_4) & = 0, &
\y(p_1) + \y(p_3) + \y(p_5) & = 0,  \\
\y(p_2) + \y(p_4) + \y(p_6) & = 0,  &
\y(p_1) + \y(p_2) + \y(p_6) & = 0,  &
\y(p_3) + \y(p_5) + b & = 0. 
\end{aligned}
\]
The above equations imply $\y(p_3) = \y(p_6) = a$, $\y(p_1) = \y(p_4) = b$ and
$\y(p_2) = \y(p_5) = - a - b$.
\end{proof}

\begin{lemma}
\label{lem:batLemma}
Let $G$ be a nut graph, let $uv \in E(G)$, and let $\x \in \ker G$ with $\x \neq
0$. Then $\mathrm{PS}(G, uv)$ is a nut graph if and only if\, $\x(u) + \x(v)
\neq 0$.
\end{lemma}

\begin{proof}
Define $\widetilde{\x} \colon V(\mathrm{PS}(G, uv)) \to \mathbb{R}$ by
\begin{equation}
\widetilde{\x}(w) = \begin{cases}
\x(w), & \text{if } w \in V(G); \\
\x(u), & \text{if } w \in \{ p_3, p_6 \}; \\
\x(v), & \text{if } w \in \{ p_1, p_4 \}; \\
 - \x(u) - \x(v), & \text{if } w \in \{ p_2, p_5 \};
\end{cases}
\end{equation}
where the extra vertices of $\mathrm{PS}(G, uv)$ are labelled $p_1, \ldots, p_6$
as in Lemma~\ref{lem:batEigenvectors}. It is straightforward to check that
$\widetilde{\x} \in \ker \mathrm{PS}(G, uv)$. It remains to show that
$\eta(\mathrm{PS}(G, uv)) \leq 1$.

To obtain a contradiction, suppose that there exists another eigenvector
$\widetilde{\y} \in \ker \mathrm{PS}(G, uv)$ that is linearly independent of
$\widetilde{\x}$. Consider the vector $\widetilde{\z} = \alpha \widetilde{\x} -
\widetilde{\y}$, where $\alpha = \widetilde{\y}(u)/\widetilde{\x}(u)$. Note that
$\widetilde{\z} \in \ker \mathrm{PS}(G, uv)$, $\widetilde{\z}$ is non-trivial,
and $\widetilde{\z}(u) = 0$. By Lemma~\ref{lem:batEigenvectors},
$\widetilde{\z}(u) = \widetilde{\z}(p_6)$ and $\widetilde{\z}(v) =
\widetilde{\z}(p_4)$. The restriction of $\widetilde{\z}$ to graph $G$ is a
kernel eigenvector of $G$ with zero entry on $u$, which contradicts the fact
that $G$ is a nut graph. Hence, $\widetilde{\y}$ does not exist and
$\eta(\mathrm{PS}(G, uv)) \leq 1$.

The graph $\mathrm{PS}(G, uv)$ is a nut graph if $\widetilde{\x}$ is full, and
$\widetilde{\x}$ is full if and only if $- \x(u) - \x(v) \neq 0$.
\end{proof}

\begin{lemma}
\label{lem:goodBridgeLemma}
Let $G$ be a nut graph, let $\x \in \ker G$, and let $uv \in E(G)$ be a
 \emph{good} bridge, i.e., a bridge such that $\x(u) \neq -\x(v)$. Then
 $\mathrm{PS}(G, uv)$ is a nut graph with $\varphi(\mathrm{PS}(G, uv)) = \varphi
 (G) + 1$. Moreover, $\mathrm{PS}(G, uv)$ contains at least two good bridges.
\end{lemma}

\begin{proof}
Let the extra vertices of $\mathrm{PS}(G, uv)$ be labelled $p_1, \ldots, p_6$
(as in Lemma~\ref{lem:batEigenvectors}). By Lemma~\ref{lem:batLemma},
$\mathrm{PS}(G, uv)$ is a nut graph. Moreover, edges $up_4$ and $vp_6$ are good
bridges. Since $uv$ is a bridge in $G$, the graph $G$ can be obtained by
coalescing $G_1$, $G_2$ and $K_2$ (i.e.\ the bridge $uv$), where $G_1$ is the
subgraph of $G$ induced on the vertices that are closer to $u$ than to $v$ and
the graph $G_2$ is the subgraph of $G$ induced on the vertices that are closer
to $v$ than to $u$; see Figure 7(a). By Corollary~\ref{lem:coalAdditive},
$\varphi(G) = \varphi(G_1) + \varphi(G_2)$. On the other hand, graph
$\mathrm{PS}(G, uv)$ can be obtained by coalescing $G_1$, $G_2$, two copies of
$K_2$ (bridges $up_4$ and $vp_6$) and the graph $P$, where $P$ denotes the
subgraph of $\mathrm{PS}(G, uv)$ induced on the vertices $p_1, \ldots, p_6$; see
Figure~7(b). By Corollary~\ref{lem:coalAdditive}, $\varphi(\mathrm{PS}(G, uv)) =
\varphi(G_1) + \varphi(G_2) + \varphi(P)$. It is easy to check that $\varphi(P)
= 1$, and the result follows.
\end{proof}

\begin{figure}[!htp]
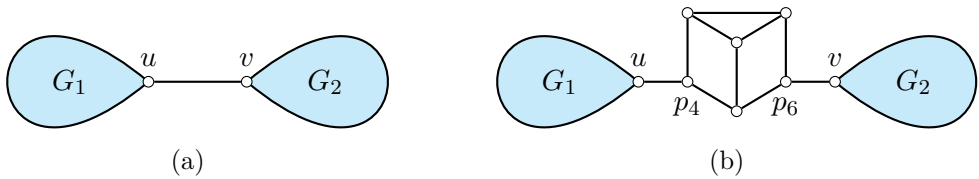

\centering
\subcaptionbox{}{
\input{img/prism_splice_panela.tikz}
}
\qquad
\subcaptionbox{}{
\input{img/prism_splice_panelb.tikz}
}
\caption{A graph $G$ with a bridge $uv$ and the prism splice $\mathrm{PS}(G,
uv)$.}
\end{figure}

\noindent Hence we have a theorem for bipartite edge frustration of chemical
graphs:

\begin{theorem}
\label{prop:chemicalRealm}
For every integer $k \geq 2$ there exists a chemical nut graph $G$, such that
$\varphi(G) = k$.
\end{theorem}

\begin{proof}
Let $G_0$ be the smallest chemical nut graph
(Figure~\ref{fig:frustratedChemicalsBridged}(a)). $G_0$ contains a good bridge
$uv$ and we note that $\varphi(G_0) = 2$. By Lemma~\ref{lem:goodBridgeLemma},
the graph $G_1 \coloneqq \mathrm{PS}(G_0, uv)$ is a nut graph with $\varphi(G_1)
= \varphi(G_0) + 1$. Since $G_1$ also contains a good bridge, applying
Lemma~\ref{lem:goodBridgeLemma} iteratively yields a sequence of chemical nut
graphs $G_0, G_1, G_2, G_3, \ldots$ such that $\varphi(G_i) = i + 2$.
\end{proof}

\begin{figure}[!htbp]
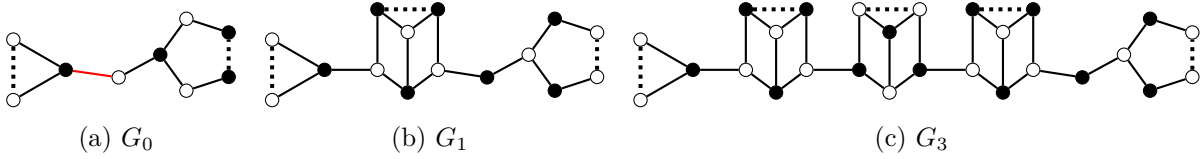

\centering
\subcaptionbox{$G_0$}{
\input{img/theorem21_panela.tikz}
}
\subcaptionbox{$G_1$}{
\input{img/theorem21_panelb.tikz}
}
\subcaptionbox{$G_3$}{
\input{img/theorem21_panelc.tikz}
}
\caption{Graphs $G_0$, $G_1$ and $G_3$ from the proof of
Theorem~\ref{prop:chemicalRealm}. The red edge indicates a good bridge.}
\label{fig:frustratedChemicalsBridged}
\end{figure}

Note that, in the proof of Theorem~\ref{prop:chemicalRealm} above, applying the
prism splice construction to a `seed' graph increases the bipartite edge
frustration by $1$ and the order by $6$. As $G_0$ is of odd order, all graphs
$G_0, G_1, G_2, \ldots$ are of odd order. At any stage, the bridge
construction~\cite{Gutman1998} can be used to further increase order by $2$,
leaving bipartite edge frustration unchanged. We may also produce examples of
even order for all $k \geq 2$ by choosing a `seed' graph $G_0'$ such as the
graph in Figure~\ref{fig:alternativeSeeds}(a). Observe that the seed graph
$G_0''$ from Figure~\ref{fig:alternativeSeeds}(b) yields a sequence of cubic
graphs, and hence the following corollary.

\begin{corollary}
\label{thm:cubicAnyFrus}
For every integer $k \geq 2$ there exists a cubic nut graph $G$, such that
$\varphi(G) = k$.
\end{corollary}

\begin{figure}[!htp]
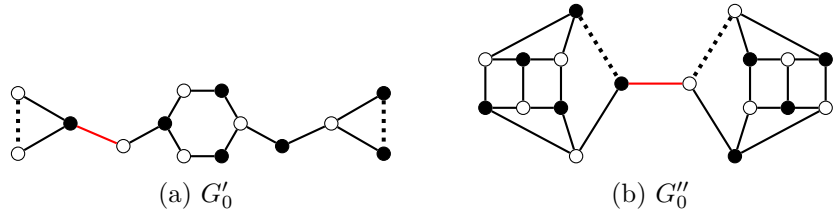

\centering
\subcaptionbox{$G_0'$}{
\input{img/fig9_panela.tikz}
}
\qquad
\subcaptionbox{$G_0''$}{
\input{img/fig9_panelb.tikz}
}
\caption{A seed graph $G_0'$ of even order and a cubic seed graph $G_0''$.}
\label{fig:alternativeSeeds}
\end{figure}

\subsection{Regular nut graphs with prescribed bipartite edge frustration}
\label{sec:5dot4}

Now we turn to regular nut graphs. We have seen that in the regular case, the
admissible values of frustration may be limited by the parity of the degree and
order (see Proposition~\ref{prop:explainMissingFrust}). The cubic case is
settled by Corollary~\ref{thm:cubicAnyFrus}. The next theorem shows that in the
quartic case, all admissible frustrations may be obtained.

\begin{theorem}
\label{thm:quarticAnyFrus}
For every even $k \geq 2$ there exists a quartic nut graph $G$ with $\varphi(G)
= k$. Moreover, there exists no quartic nut graph $G$ with odd $\varphi(G)$.
\end{theorem}

\noindent For $\ell \geq 3$, the antiprism graph, $A_\ell$, has vertex set
$V(A_\ell) = \{ u_0, \ldots, u_{\ell - 1}, v_0, \ldots, v_{\ell - 1} \}$, where
\linebreak $u_i \sim u_{i + 1}, v_i \sim v_{i + 1}, u_i \sim v_i$ and $u_i \sim
v_{i + 1}$ for all $i$ (all indices modulo $\ell$). In other words, $A_\ell$ is
the skeleton of an $\ell$-gonal antiprism. In \cite{GPS}, the following was
proved.

\begin{proposition}[{\cite{GPS}}]
\label{prop:gpsProp}
The antiprism graph $A_\ell$ of order $2\ell$ is a nut graph if and only if
$2\ell \not\equiv 0 \pmod{6}$.
\end{proposition}

\par\noindent Noting that antiprisms are $4$-regular graphs, we first establish
the following lemma.

\begin{lemma}
\label{lem:antiprismFrustration}
Let $A_\ell$ be the antiprism graph of order $2\ell$. Then
\begin{equation}
\varphi(A_\ell) = \begin{cases}
\ell & \text{if }\ell \equiv 0 \pmod{2}, \\
\ell + 1 & \text{if }\ell \not\equiv 0 \pmod{2}.
\end{cases}
\end{equation}
\end{lemma}

\begin{proof}
If $\ell$ is even, then $A_\ell^*$ has $2\ell$ odd-degree vertices and there
exists a matching $M$ that covers them. By Lemma~\ref{lem:hadlock},
$\varphi(A_\ell) = |M| = \ell$. If $\ell$ is odd, then $A_\ell^*$ has $2\ell +
2$ odd-degree vertices and there exists a matching $M$ that covers them. By
Lemma~\ref{lem:hadlock}, $\varphi(A_\ell) = |M| = \ell + 1$. (For illustrations
of $A_\ell^*$ and the matchings $M$, see Figure~\ref{fig:antiprismDuals}.)
\end{proof}

\begin{figure}[!ht]
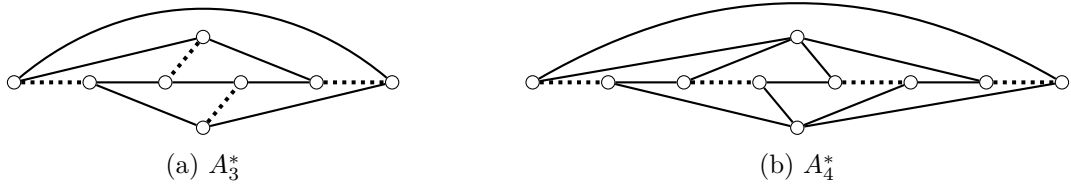

\centering
\subcaptionbox{$A_3^*$}{\input{img/antiprism_dual_matching1.tikz}}
\qquad\qquad
\subcaptionbox{$A_4^*$}{\input{img/antiprism_dual_matching2.tikz}}
\caption{Duals of antiprisms of odd and even $\ell$ and the corresponding
matchings (indicated by dotted lines). }
\label{fig:antiprismDuals}
\end{figure}

\begin{figure}[!b]
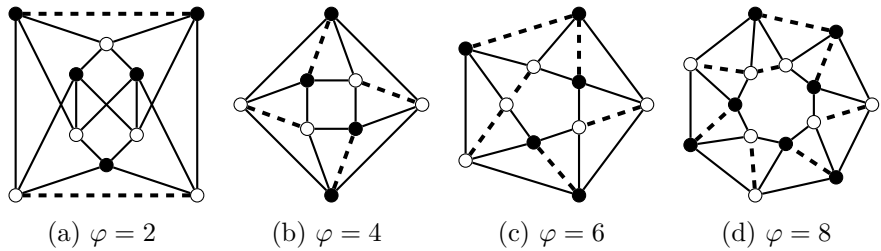

\centering
\subcaptionbox{$\varphi=2$}{ \input{img/fig13a.tikz} }
\subcaptionbox{$\varphi=4$}{ \input{img/fig13b.tikz} }
\subcaptionbox{$\varphi=6$}{ \input{img/fig13c.tikz} }
\subcaptionbox{$\varphi=8$}{ \input{img/fig13d.tikz} }
\caption{Examples of quartic nut graphs with bipartite edge frustration from the
set $\{2, 4, 6, 8\}$.}
\label{fig:smallQuarticFrustrated}
\end{figure}

\begin{proof}[Proof of Theorem~\ref{thm:quarticAnyFrus}]
For $k = 2$, existence is settled by the example in
Figure~\ref{fig:smallQuarticFrustrated}(a). Suppose $k \geq 4$. There are three
cases:
\begin{enumerate}[label=(\alph*)]
\item If $k \equiv 0 \pmod{6}$, take the antiprism $A_{k - 1}$. As $2k - 2
\equiv 4 \pmod{6}$, $A_{k - 1}$ is a nut graph by
Proposition~\ref{prop:gpsProp}. As $k - 1$ is odd, from
Lemma~\ref{lem:antiprismFrustration} we obtain $\varphi(A_{k - 1}) = k$.
\item If $k \equiv 2 \pmod{6}$, take the antiprism $A_{k - 1}$. As $2k - 2
\equiv 2 \pmod{6}$, $A_{k - 1}$ is a nut graph by
Proposition~\ref{prop:gpsProp}. As $k - 1$ is odd, from
Lemma~\ref{lem:antiprismFrustration} we obtain $\varphi(A_{k - 1}) = k$.
\item If $k \equiv 4 \pmod{6}$, take the antiprism $A_k$. As $2k \equiv 2
\pmod{6}$, $A_k$ is a nut graph by Proposition~\ref{prop:gpsProp}. As $k$ is
even, from Lemma~\ref{lem:antiprismFrustration} we obtain $\varphi(A_k) = k$.
\qedhere
\end{enumerate}
\end{proof}

There is an alternative proof strategy similar to the one used in
\cite{AutomorphismPaper}, which we show here, as this approach can be used for
higher degrees. For the quartic case, we will need only two `building blocks',
nut graphs $G^\text{\,term}$ and $G^\text{\,link}$, where $G^\text{\,term}$ has
one degree-$2$ vertex and all others of degree $4$, but $G^\text{\,link}$ has
two degree-$2$ vertices and all others of degree $4$. We can obtain a larger
quartic nut graph by pairwise coalescence of blocks (shown schematically in
Figure~\ref{fig:schematicSausage}) and exploiting the additivity of bipartite
edge frustration under this operation (see Corollary~\ref{lem:coalAdditive}).

\begin{figure}[!htbp]
\centering
\input{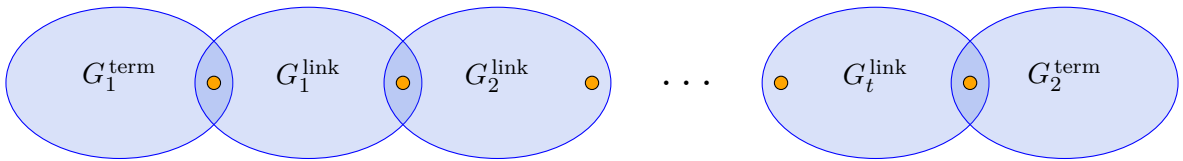}
\caption{Schematic depiction of a regular nut graph obtained by chaining
building blocks $G^\text{\,term}_1, G^\text{\,link}_1, \ldots,
G^\text{\,link}_t$ and $G^\text{\,term}_2$ using the coalescence operation.}
 \label{fig:schematicSausage}
\end{figure}

\begin{proof}[Alternative proof of Theorem~\ref{thm:quarticAnyFrus}]
Let $G_{1}$ be the graph shown in Figure~\ref{fig:quarticBuildingBlocks}(a).
$G_1$ is a nut graph that has 13 vertices of degree 4 and one vertex of degree
2. Moreover, $\varphi(G_{1}) = 5$. Let $G_2$ be the graph shown in
Figure~\ref{fig:quarticBuildingBlocks}(b). $G_2$ is a nut graph with 7 vertices
of degree 4 and two vertices of degree 2. Moreover, $\varphi(G_2) = 2$.
\begin{figure}[!htbp]
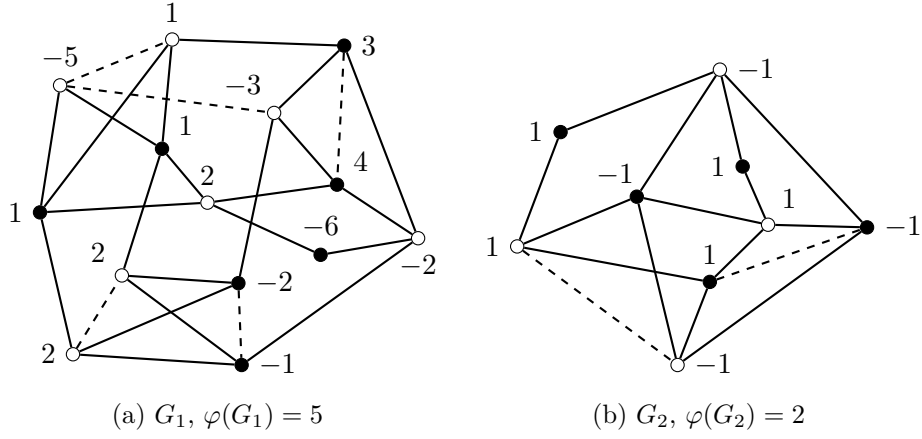

\centering
\subcaptionbox{$G_1$, $\varphi(G_1)=5$}{
\input{img/quartic_block_g1.tikz}
}
\subcaptionbox{$G_2$, $\varphi(G_2)=2$}{
\input{img/quartic_block_g2.tikz}
}
\caption{The two graphs that serve as building blocks in the proof of
Theorem~\ref{thm:quarticAnyFrus}. Vertex labels indicate kernel eigenvector
entries. The given black-and-white colouring minimises the number of frustrated
edges in each graph. Frustrated edges are denoted with dashed lines. Coalescence
is via the vertices of degree $2$.}
\label{fig:quarticBuildingBlocks}
\end{figure}
The graphs $G_1$ and $G_2$ can serve as the blocks $G^\text{\,term}$ and
$G^\text{\,link}$, respectively. We can obtain any even frustration $k \geq 10$
by chaining $(k - 10) / 2$ copies of $G_2$ and terminating the chain with $G_1$
at each end. Table~\ref{tbl:quarticFrustration} settles the cases $k \in \{2, 4,
6, 8\}$. Examples of such graphs are shown in
Figure~\ref{fig:smallQuarticFrustrated}.
\end{proof}

The same strategy will also be used in characterising bipartite edge
frustrations that can be realised by quintic and sextic nut graphs.

\begin{theorem}
\label{thm:quinticAnyFrus}
For $k = 2$ and for every $k \geq 4$ there exists a quintic nut graph $G$ with
$\varphi(G) = k$. Moreover, there exists no quintic nut graph $G$ with
$\varphi(G) = 3$.
\end{theorem}

\begin{proof}
Let $L_5, T_2^{(0)}, T_2^{(1)}, T_2^{(2)}$ and $T_3$ be graphs given by
\texttt{graph6}~\cite{graph6} strings in Table~\ref{tbl:quinticGraph6}. For $k
\geq 8$, we start the chain with $T_3$, followed by $\lfloor \frac{k - 8}{3}
\rfloor$ copies of $L_5$. We terminate the chain with $T_2^{(k \bmod 3)}$. The
graph $G$ obtained in this way is $5$-regular by design. As $L_5, T_2^{(0)},
T_2^{(1)}, T_2^{(2)}$ and $T_3$ are all nut graphs, $G$ is a nut graph by
Lemma~\ref{lem:coreCoalescenceNullity}. By Corollary~\ref{lem:coalAdditive},
$\varphi(G) = \varphi(T_3) + \lfloor \frac{k - 8}{3} \rfloor \varphi(L_5) +
\varphi(T_2^{(k \bmod 3)}) = k$. Table~\ref{tbl:quinticFrustration} settles the
cases where $k \in \{2, 4, 5, 6, 7\}$. The case $\varphi = 3$ is impossible for
a quintic nut graph (Proposition~\ref{prop:explainMissingFrust}).
\end{proof}

\begin{table}[!ht]
\centering
\renewcommand{\arraystretch}{1.4}
{\small
\begin{tabular}{|l|l|l|c|}
\hline
$H$ & \texttt{graph6} string & Degrees & $\varphi(H)$ \\
\hline
\hline
$L_5$ & \verb:J?BVdzw\f??: & $\{ 2^1, 3^1, 5^9 \}$ & 3 \\
$T_2^{(0)}$ & \verb:N??FEbKelWFg^?\_?H?: & $\{ 2^1, 5^{14} \}$ & 4 \\
$T_2^{(1)}$ & \verb:N??EDbkvDKN_MoTo@@?: & $\{ 2^1, 5^{14} \}$ & 5 \\
$T_2^{(2)}$ & \verb:N??C@qfNfg\_v?UoGO?: & $\{ 2^1, 5^{14} \}$ & 3 \\
$T_3$ & \verb:M??MF?|kBsMof_rO?: & $\{ 3^1, 5^{13} \}$ & 5 \\
\hline
\end{tabular}%
}
\caption{Links and terminals used for constructing quintic nut graphs with
$\varphi \geq 8$.}
\label{tbl:quinticGraph6}
\end{table}

\begin{theorem}
For every $k \geq 2$ there exists a sextic nut graph $G$ with $\varphi(G) = k$.
\end{theorem}

\begin{proof}
Let $L_6, T_2^{(0)}, T_2^{(1)}, T_4^{(0)}$ and $T_4^{(1)}$ be graphs given by
\texttt{graph6} strings in Table~\ref{tbl:sexticGraph6}. For $k \geq 11$, we
start the chain with $T_4^{(k \bmod 2)}$, followed by $\lfloor \frac{k - 11}{4}
\rfloor$ copies of $L_6$. We terminate the chain with $T_2^{(\lceil k/2 \rceil
\bmod 2)}$. The graph $G$ obtained in this way is $6$-regular by design. As
$L_6, T_2^{(0)}, T_2^{(1)}, T_4^{(0)}$ and $T_4^{(1)}$ are all nut graphs, $G$
is a nut graph by Lemma~\ref{lem:coreCoalescenceNullity}. By
Corollary~\ref{lem:coalAdditive}, $\varphi(G) = \varphi(T_4^{(k \bmod 2)}) +
\lfloor \frac{k - 11}{4} \rfloor \varphi(L_6) + \varphi(T_2^{(\lceil k/2 \rceil
\bmod 2)}) = k$. Table~\ref{tbl:sexticFrustration} settles the cases where $k
\in \{ 3, 5, 6, 7, 8, 9, 10\}$. We provide graphs $G_6^{(2)}$ and $G_6^{(4)}$
which exemplify the cases $k = 2$ and $k = 4$.
\end{proof}

\begin{table}[!hbt]
\centering
\renewcommand{\arraystretch}{1.4}
{\small
\begin{tabular}{|l|l|l|c|}
\hline
$H$ & \texttt{graph6} string & Degrees & $\varphi(H)$ \\
\hline
\hline
$L_6$ & \verb:L??FUyy~FsNK{?: & $\{ 2^1, 4^1, 6^{11} \}$ & 4 \\
$T_2^{(0)}$ & \verb:M?AFrzszEx^ONC?O_: & $\{ 2^1, 6^{13} \}$ & 4 \\
$T_2^{(1)}$ & \verb:M?AFbZ]nby]_}OA@?: & $\{ 2^1, 6^{13} \}$ & 6 \\
$T_4^{(0)}$ & \verb:JCrvbql~Br?: & $\{ 4^1, 6^{10} \}$ & 8 \\
$T_4^{(1)}$ & \verb:KCQfex{zTwND: & $\{ 4^1, 6^{11} \}$ & 7 \\
\hline
\hline
$G_6^{(2)}$ & \verb:O_???A{[}]Vo]ozORwBm?: & $\{ 6^{16} \}$ & 2 \\
$G_6^{(4)}$ & \verb:O??_AB{{s]Ro]oyPRwBkO: & $\{ 6^{16} \}$ & 4 \\
\hline
\end{tabular}%
}
\caption{Links and terminals used for constructing sextic nut graphs with
$\varphi \geq 11$, and two ancillary graphs used for completion of the proof.}
\label{tbl:sexticGraph6}
\end{table}

\begin{theorem}
For $k \in \{ 2, 4 \}$ and for every $k \geq 6$ there exists a septic nut graph
$G$ with $\varphi(G) = k$. Moreover, there exists no septic nut graph $G$ with
$\varphi(G) \in \{ 3, 5 \}$.
\end{theorem}

\begin{proof}
Let $L_7, T_2^{(0)}, T_2^{(1)}, T_5^{(0)}, T_5^{(1)}, T_5^{(2)}$ and $T_5^{(3)}$
be graphs given by \texttt{graph6} strings in Table~\ref{tbl:septicGraph6}. For
$k \geq 17$, we start the chain with $T_2^{(k \bmod 2)}$, followed by $\lfloor
\frac{k - 17}{8} \rfloor$ copies of $L_7$. We terminate the chain with
$T_5^{(\lceil k/2 \rceil \bmod 4)}$. The graph $G$ obtained in this way is
$7$-regular by design. As $L_7, \ldots, T_5^{(3)}$ are all nut graphs, $G$ is a
nut graph by Lemma~\ref{lem:coreCoalescenceNullity}. By
Corollary~\ref{lem:coalAdditive}, $\varphi(G) = \varphi(T_2^{(k \bmod 2)}) +
\lfloor \frac{k - 17}{8} \rfloor \varphi(L_7) + \varphi(T_5^{(\lceil k/2 \rceil
\bmod 4)}) = k$. Table~\ref{tbl:septicFrustration} settles the cases where $k
\in \{ 8, 9, \ldots, 16 \}$. We provide graphs $G_7^{(k)}$ in
Table~\ref{tbl:septicGraph6} which exemplify the cases $k \in \{ 2, 4, 6, 7 \}$.
The cases $\varphi = 3$ and $\varphi = 5$ are impossible for a septic nut graph
(Proposition~\ref{prop:explainMissingFrust}).
\end{proof}

\begin{table}[!hbt]
\centering
\renewcommand{\arraystretch}{1.4}
{\small
\begin{tabular}{|l|l|l|c|}
\hline
$H$ & \texttt{graph6} string & Degrees & $\varphi(H)$ \\
\hline
\hline
$L_7$ & \verb:N_E]u?NBoM~BVY{YYE?: & $\{ 2^1, 5^1, 7^{13} \}$ & 8  \\
$T_2^{(0)}$ & \verb:P???Bbkfq}\[|gnc^oB{O_O?: & $\{ 2^1, 7^{16} \}$ & 7  \\
$T_2^{(1)}$ & \verb:P???FaM^by|wvo}g]sD{??GC: & $\{ 2^1, 7^{16} \}$ & 6  \\
$T_5^{(0)}$ & \verb:OQ{TsDhG[irsPUPEweYmI: & $\{ 5^1, 7^{15} \}$ & 17  \\
$T_5^{(1)}$ & \verb:OCUjUWua~?ix{@uGATWIv: & $\{ 5^1, 7^{15} \}$ & 11  \\
$T_5^{(2)}$ & \verb:MFm@]NeVBRVqT[ws_: & $\{ 5^1, 7^{13} \}$ & 13  \\
$T_5^{(3)}$ & \verb:M{VAhkwYlRnsb]do?: & $\{ 5^1, 7^{13} \}$ & 15  \\
\hline
\hline
$G_7^{(2)}$ & \verb:S?G?????F{]MxWysY{C|_N[@yoA}_CNc?: & $\{ 7^{20} \}$ & 2  \\
$G_7^{(4)}$ & \verb:S????B??D{Ng}g{[W}C]oNs@joBkgCz_?: & $\{ 7^{20} \}$ & 4  \\
$G_7^{(6)}$ & \verb:S_???A?C@{^gyg{[WzC^_N[?zoBNOA{c?: & $\{ 7^{20} \}$ & 6  \\
$G_7^{(7)}$ & \verb:Q????Bwpvw]Xb{NkNgB]?uoIv?O: & $\{ 7^{18} \}$ & 7  \\
\hline
\end{tabular}%
}
\caption{Links and terminals used for constructing septic nut graphs with
$\varphi \geq 17$, and four ancillary graphs used for completion of the proof.}
\label{tbl:septicGraph6}
\end{table}

\begin{theorem}
\label{thm:octicAnyFrus}
For every even $k \geq 2$ there exists an octic nut graph $G$ with $\varphi(G) =
k$. Moreover, there exists no octic nut graph $G$ with odd $\varphi(G)$.
\end{theorem}

\begin{proof}
Let $L_8, T_2, T_6^{(0)}, T_6^{(1)}, T_6^{(2)}, T_6^{(3)}$ and $T_6^{(4)}$ be
graphs given by \texttt{graph6} strings in Table~\ref{tbl:octicGraph6}. For $k
\geq 28$, we start the chain with $T_2$, followed by $\lfloor \frac{k - 28}{10}
\rfloor$ copies of $L_8$. We terminate the chain with $T_6^{(\lceil k/2 \rceil
\bmod 5)}$. The graph $G$ obtained in this way is $8$-regular by design. As
$L_8, \ldots, T_6^{(4)}$ are all nut graphs, $G$ is a nut graph by
Lemma~\ref{lem:coreCoalescenceNullity}. By Corollary~\ref{lem:coalAdditive},
$\varphi(G) = \varphi(T_2) + \lfloor \frac{k - 28}{10} \rfloor \varphi(L_8) +
\varphi(T_6^{(\lceil k/2 \rceil \bmod 5)}) = k$.
Table~\ref{tbl:octicFrustration} settles the cases where $k \in \{ 12, 14,
\ldots, 22 \}$. We provide graphs $G_8^{(k)}$ in Table~\ref{tbl:octicGraph6}
which exemplify the cases $k \in \{ 2, 4, 6, 8, 10, 24, 26 \}$. Odd values of
$k$ are ruled out by Proposition~\ref{prop:explainMissingFrust}.
\end{proof}

\begin{table}[!hbt]
\centering
\renewcommand{\arraystretch}{1.4}
{\small
\begin{tabular}{|l|l|l|c|}
\hline
$H$ & \texttt{graph6} string & Degrees & $\varphi(H)$ \\
\hline
\hline
$L_8$ & \verb:NruT~UPD^Nlu?@^Kle?: & $\{ 2^1, 6^1, 8^{13} \}$ & 10  \\
$T_2$ & \verb:NwfZlqRsjuLn}SJm_c?: & $\{ 2^1, 8^{14} \}$ & 15  \\
$T_6^{(0)}$ & \verb:LeZL|~mmPxmVlV: & $\{ 6^1, 8^{12} \}$ & 15  \\
$T_6^{(1)}$ & \verb:L{]tY^F]vx[]F^: & $\{ 6^1, 8^{12} \}$ & 17  \\
$T_6^{(2)}$ & \verb:Lfh}X]~\}a\J}b: & $\{ 6^1, 8^{12} \}$ & 19  \\
$T_6^{(3)}$ & \verb:N[|Tq^VTRxO}lTyEcRg: & $\{ 6^1, 8^{14} \}$ & 21  \\
$T_6^{(4)}$ & \verb:KDfnXn|}vu~k: & $\{ 6^1, 8^{11} \}$ & 13  \\
\hline
\hline
$G_8^{(2)}$ & \verb:S????G??F}^e{{rkV{B|_}sBloEfa@~_?: & $\{ 8^{20} \}$ & 2  \\
$G_8^{(4)}$ & \verb:Q??O??_~vtZk}w|w\{E~?nw@nK?: & $\{ 8^{18} \}$ & 4  \\
$G_8^{(6)}$ & \verb:Q??__?C~vt^k}wxqX{ErBnw@~_?: & $\{ 8^{18} \}$ & 6  \\
$G_8^{(8)}$ & \verb:QB?@??O~v|Zkuw|W\{C^Gjodn`?: & $\{ 8^{18} \}$ & 8  \\
$G_8^{(10)}$ & \verb:OCc@?Jt~v\M|}ozozwB}@: & $\{ 8^{16} \}$ & 10  \\
$G_8^{(24)}$ & \verb:Q`N@oc@vVSscUjXX\{CjIfk``lO: & $\{ 8^{18} \}$ & 24  \\
$G_8^{(26)}$ & \verb:SoEAAFBHFaVfska{f`Bw@MoKLWxNcA{aC: & $\{ 8^{20} \}$ & 26  \\
\hline
\end{tabular}%
}
\caption{Links and terminals used for constructing octic nut graphs with
$\varphi \geq 28$, and seven ancillary graphs used for completion of the proof.}
\label{tbl:octicGraph6}
\end{table}

\subsection{A construction that preserves bipartite edge frustration}
\label{subsec:fowler}

Construction of larger nut graphs from smaller has been a topic of interest
since the first paper \cite{Gutman1998}. One strategy proceeds by subdividing
chosen sets of edges of the parent nut graph $G$. The construction in {\S}{4.1}
of \cite{Gutman1998} (we will call it the subdivision construction) inserts four
vertices into each chosen edge and can be applied to any edge of $G$. The
construction in {\S}{4.3} of the same paper (the bridge construction) inserts
two vertices into a bridge in $G$. As the next proposition shows, both of these
constructions preserve bipartite edge frustration.

\begin{proposition}
\label{prop:bridgeSubdivPreserve}
The following statements hold.
\begin{enumerate}[label=(\alph*)]
\item Let $G$ be a graph and let $e = uv \in E(G)$. Let $\mathcal{S}(G, e)$ be
the graph obtained from $G$ by subdividing the edge $e$ with four new vertices.
Then $\varphi(\mathcal{S}(G, e)) = \varphi(G)$.
\item Let $G$ be a graph and let $e = uv \in E(G)$ be a bridge. Let
$\mathcal{B}(G, e)$ be the graph obtained from $G$ by subdividing the bridge $e$
with two new vertices. Then $\varphi(\mathcal{B}(G, e)) = \varphi(G)$.
\end{enumerate}
\end{proposition}

\noindent The proof is straightforward and thus not given here. Another
versatile construction for producing a larger nut graph from a smaller was
introduced in~\cite{GPS}. It is defined below.

\begin{definition}
\label{def:fowler}
Let $G$ be a graph and let $v \in V(G)$ be a vertex of degree $d$ in $G$. Let
$G(v) = \{u_1, \ldots, u_d\}$. The graph $F(G, v)$, is obtained from $G$ in the
following way: (a) edges incident to $v$ are deleted; (b) let $w_1, \ldots, w_d$
and $x_1, \ldots, x_d$ denote $2d$ newly added vertices; (c) new edges are added
such that $x_i \sim u_i$ for $i = 1, \ldots, d$; and $x_i \sim w_j$ for $i \neq
j$, $1 \leq i, j \leq d$; and $w_i \sim v$ for all $i = 1, \ldots, d$.
\end{definition}

\par\noindent The construction $F(G, v)$ is sometimes called `the Fowler
construction' in the nut-graph literature. If $G$ is a nut graph and $v$ is an
arbitrary vertex, then $F(G,v)$ is also a nut graph~\cite{GPS}. The construction
has the nice property that it produces a regular graph $F(G, v)$ from a regular
graph $G$. As we will show next, it not only preserves regularity, it also
preserves frustration. We note in passing that if $v$ is a vertex of degree $2$
then $F(G, v)$ is isomorphic to $\mathcal{S}(G, e)$, where $e$ is either of the
two edges incident to $v$.

\begin{lemma}
\label{lem:fowlerFrustration}
Let $G$ be a simple graph and let $v \in V(G)$. Then $\varphi(G) = \varphi(F(G,
v))$.
\end{lemma}

\begin{proof}
If $\deg(v) < 2$ then the claim follows immediately. If $\deg(v) = 2$ then $F(G,
v) \cong \mathcal{S}(G, e)$, where $e$ is an edge incident with $v$. This case
is thus covered by Proposition~\ref{prop:bridgeSubdivPreserve}. We assume from
now on that $\deg(v) \geq 3$.

We use the notation of Definition~\ref{def:fowler}. Let $c \colon G(v) \to \{ b,
w \}$ be a colouring of the neighbourhood of vertex $v$. (We use two colours,
black and white, denoted $b$ and $w$ for brevity.) Let $d_b(c) = |c^{-1}(b)|$,
i.e.\ the number of neighbours of $v$ that are coloured black. Similarly, let
$d_w(c) = |c^{-1}(w)|$. We could extend the colouring $c$ to all vertices $V(G)
\setminus \{ v \}$, so that the number of frustrated edges in $G - v$ is
minimised, and let $\Phi(c)$ denote the minimum number of frustrated edges in $G
- v$ for the given colouring $c$. If we also want to minimise the number of
frustrated edges among those that are incident with $v$, it is clear that vertex
$v$ has to be coloured black if $d_b(c) < d_w(c)$, and white otherwise. The
number of frustrated edges among those where $v$ is an endvertex is then $\min\{
d_b(c), d_w(c) \}$. The frustration of the graph $G$ is then $\varphi(G) = \min
\{ \Phi(c) + \min\{ d_b(c), d_w(c) \} \mid c \in \{b, w\}^{G(v)} \}$.

\begin{figure}[!htbp]
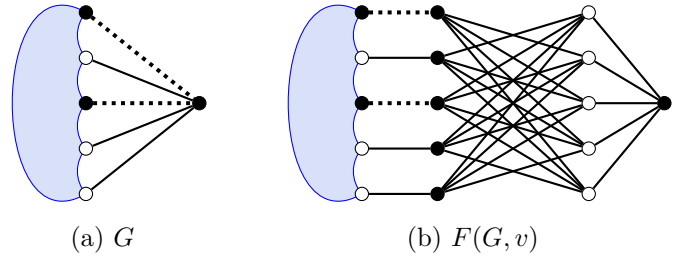

\centering
\subcaptionbox{$G$}{\input{img/fowler01.tikz}}
 \qquad
\subcaptionbox{$F(G, v)$}{\input{img/fowler02.tikz}}
\caption{Schematic illustration of extension of a colouring from $G$ to $F(G,v)$
that preserves the number of frustrated edges.}
\label{fig:fowlerColouring}
\end{figure}

Now use the same colouring $c$ on the $u$-layer in $F(G, v)$. Note that $G - v$
is isomorphic to $H \coloneqq F(G, v) - \{v, w_1, \ldots, w_d, x_1, \ldots, x_d
\}$, so the minimum number of frustrated edges in $H$ is also $\Phi(c)$. Let us
colour the entire $x$-layer in $F(G, v)$ with black if $d_b(c) < d_w(c)$, and
white otherwise. The vertex $v$ in $F(G, v)$ should then be coloured with the
same, and the $w$-layer with the opposite colour. (See
Figure~\ref{fig:fowlerColouring}.) The number of additional frustrated edges
under the colouring defined in this way is $\min\{ d_b(c), d_w(c) \}$. Provided
that this way of colouring is optimal, we have $\varphi(F(G, v)) = \min \{
\Phi(c) + \min\{ d_b(c), d_w(c) \} \mid c \in \{b, w\}^{G(v)} \} = \varphi(G)$.

It remains to show that the above is indeed an optimal colouring. First, we note
that all vertices in the $w$-layer are of the same colour. To obtain a
contradiction, suppose that two vertices, $w_i$ and $w_j$, are coloured
differently. Vertex $v$ is a neighbour of both $w_i$ and $w_j$; however we
colour $v$, we obtain at least one frustrated edge among those connecting $v$ to
the $w$-layer. Moreover, $w_i$ and $w_j$ have $d-2$ common neighbours in the
$x$-layer. However we colour these common neighbours, we obtain a further $d -
2$ frustrated edges among those connecting $x$-layer to $w$-layer. Thus, we have
at least $d - 1$ frustrated edges, but $d - 1 > \min\{ d_b(c), d_w(c) \}$ as
$\min\{ d_b(c), d_w(c) \} \leq \lfloor d / 2\rfloor$, a contradiction.

Next, we show that all vertices of the $x$-layer must have the opposite colour
to those of the $w$-layer. To obtain a contradiction, suppose that a vertex
$x_i$ has the same colour as the $w$-layer. Vertex $x_i$ has $d - 1$ neighbours
in the $w$-layer, which induces at least $d - 1$ frustrated edges, a
contradiction. The result follows.
\end{proof}

\subsection{Cubic polyhedral nut graphs}

The next result from~\cite{Fowler2021} tells us that the construction described
in {\S}\ref{subsec:fowler} preserves polyhedrality for cubic graphs.

\begin{lemma}[{\cite[Proposition 6]{Fowler2021}}]
\label{lem:fowlerPolyhedral}
Let $G$ be a cubic polyhedral graph and let $v \in V(G)$. Then $F(G, v)$ is also
a cubic polyhedral graph.
\end{lemma}

A convenient way for visualising Lemma~\ref{lem:fowlerPolyhedral} is the
following description of $F(G, v)$. Let the neighbours of $v$ be denoted $u_1,
u_2$ and $u_3$. For a cubic graph $G$, take the disjoint union of $G$ and the
cube graph $Q_3$ with vertices labelled as in Figure~\ref{fig:cubeFowler}(a).
Then $F(G, v)$ is obtained from $G \sqcup Q_3$ by deleting vertices $v$ and $u$
and adding edges $w_1u_1, w_2u_2$ and $w_3u_3$; see
Figure~\ref{fig:cubeFowler}(b).

\begin{figure}[!ht]
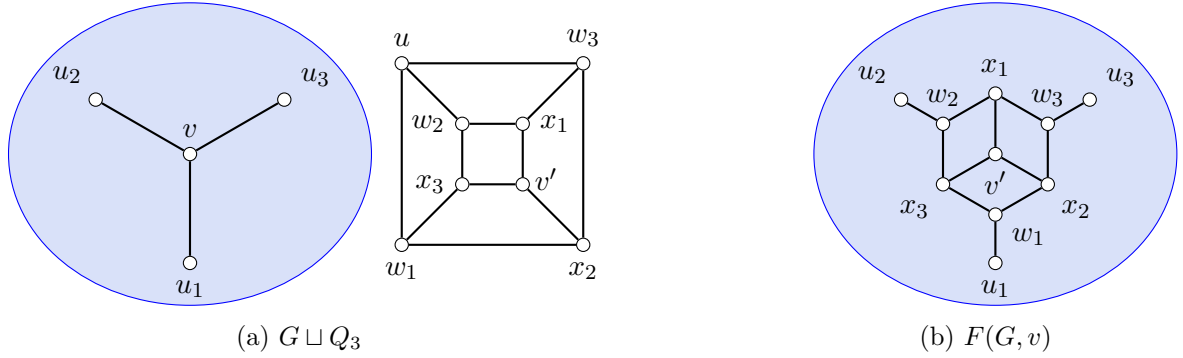

\centering
\subcaptionbox{$G \sqcup Q_3$} { 
\input{img/cube_fowler_panela.tikz}}
\qquad\qquad\qquad
\subcaptionbox{$F(G, v)$} { 
\input{img/cube_fowler_panelb.tikz}}
\caption{A visualisation of $F(G, v)$ for a cubic graph $G$.}
\label{fig:cubeFowler}
\end{figure}

\begin{theorem}
\label{thm:polyhedral}
For every $k \geq 2$ there exist infinitely many cubic polyhedral nut graphs of
bipartite edge frustration $k$.
\end{theorem}

\noindent We will provide a constructive proof. The essential step is proving
that the constructed graphs have nullity $1$, and this is done in full detail.
The remaining steps of establishing planarity, $3$-connectivity and the values
of the bipartite edge frustration are presented more compactly as they rely on
the standard machinery. Before the proof, we will define two expansions and
prove some lemmas.

\begin{figure}[!htbp]
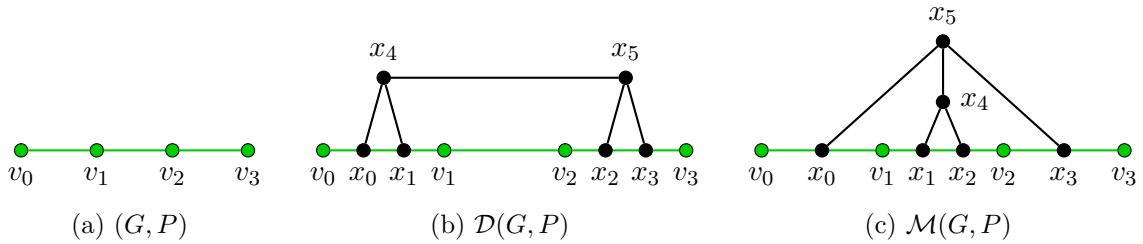

\centering
\subcaptionbox{$(G, P)$} { \input{img/path4.tikz} }
\subcaptionbox{$\mathcal{D}(G, P)$} { \input{img/path4dog.tikz} }
\subcaptionbox{$\mathcal{M}(G, P)$} { \input{img/path4man.tikz} }
\caption{The two extensions used to construct an infinite family of cubic
polyhedral nut graphs for the proof of Theorem~\ref{thm:polyhedral}.}
\end{figure}

\begin{definition}
Let $(G, P)$ be a pair where $G$ is a graph and $P \subseteq G$ is a subgraph
that is isomorphic to $P_4$. Such a pair is called a \emph{$P_4$-rooted graph}.
Suppose that $V(G) = \{v_0, v_1, \ldots, v_{n - 1}\}$ such that $P = v_0 v_1 v_2
v_3$.

Define $\mathcal{D}(G, P)$ to be the pair $(G', P')$, such that $V(G') = V(G)
\cup \{ x_0, x_1, \ldots, x_5 \}$, where $x_0, \ldots, x_5$ are new vertices,
and
\[
E(G') = (E(G) \setminus \{ v_0v_1, v_2v_3 \})
\cup \{ v_0x_0, x_0x_1, x_1v_1, v_2x_2, x_2x_3, x_3v_3,
x_0x_4, x_1x_4, x_2x_5, x_3x_5, x_4x_5 \},
\]
as shown in Figure~16(b), and $P' = x_1 v_1 v_2 x_2$.

Define $\mathcal{M}(G, P)$ to be the pair $(G'', P'')$, such that $V(G'') = V(G)
\cup \{ x_0, x_1, \ldots, x_5 \}$ and
\begin{align*}
E(G'') & = (E(G) \setminus \{ v_0v_1, v_1v_2, v_2v_3 \})  \\ 
 & \qquad {} \cup \{ v_0x_0, x_0v_1, v_1x_1, x_1x_2, x_2v_2, v_2x_3, x_3v_3,
   x_0x_5, x_1x_4,x_2x_4, x_3x_5, x_4x_5 \},
\end{align*}
as shown in Figure~16(c), and $P'' = v_1 x_1 x_2 v_2$.
\end{definition}

\begin{figure}[!htbp]
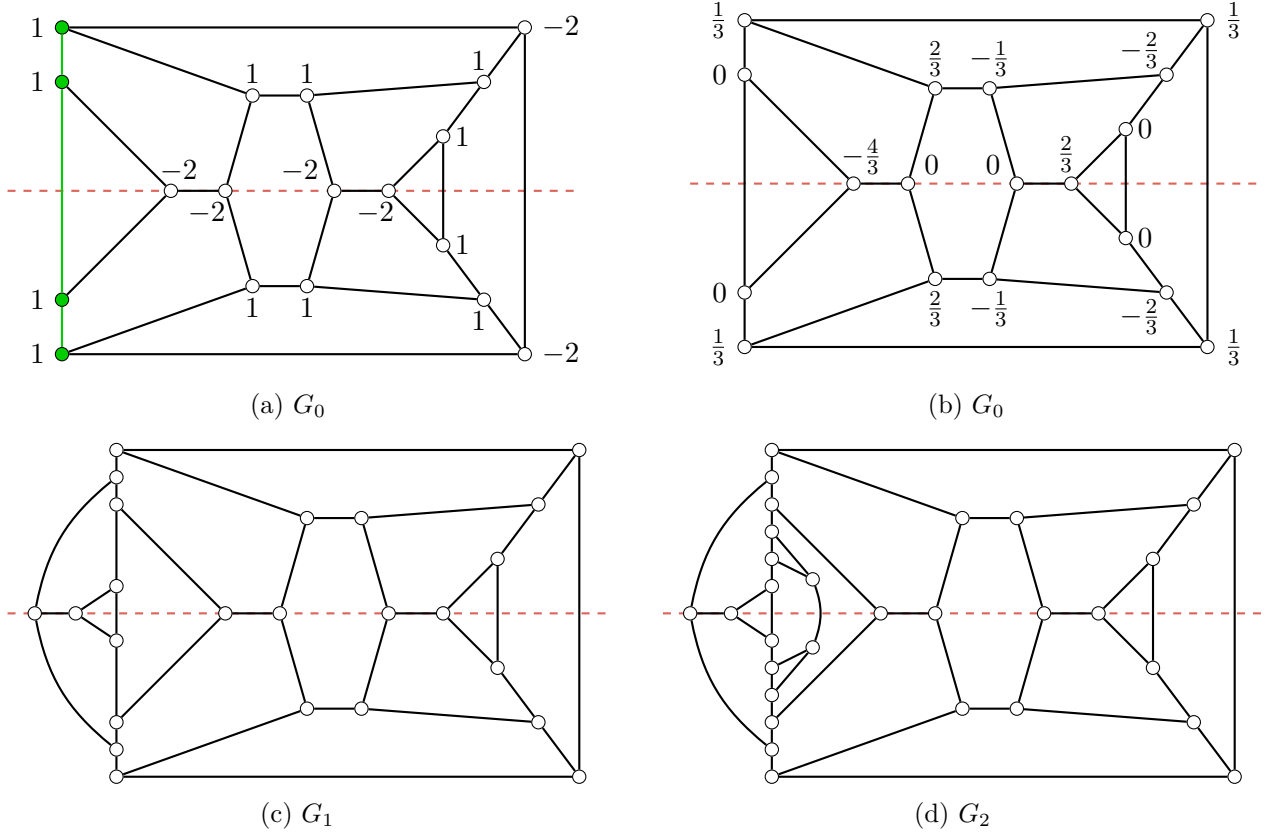

\centering
\subcaptionbox{$G_0$}{ 
\input{img/polyhedral_seed_panela.tikz}
 }
 \hfill
\subcaptionbox{$G_0$}{
\input{img/polyhedral_seed_panelb.tikz}
 }
 
 \vspace{0.5\baselineskip}
 \subcaptionbox{$G_1$}{
\input{img/polyhedral_seed_panelc.tikz}
}
\hfill
\subcaptionbox{$G_2$}{
\input{img/polyhedral_seed_paneld.tikz}
}
\caption{The first three members of the family used in the proof of
Theorem~\ref{thm:polyhedral}. The automorphism $\tau_i$ corresponds to
reflection in the dotted red line.}
\label{fig:polyhedralSeed}
\end{figure}

\begin{definition}
\label{def:manDogFamily}
Let $(G_0, P_0)$ be the $P_4$-rooted graph in
Figure~\ref{fig:polyhedralSeed}(a), with $P_0$ the path coloured green. Let
$(G_{i + 1}, P_{i + 1}) = \mathcal{M}(G_i, P_i)$ for even $i$ and $(G_{i + 1},
P_{i + 1}) = \mathcal{D}(G_i, P_i)$ for odd $i$.
\end{definition}

\begin{lemma}
For every $i \geq 0$, there exists an involution $\tau_i \in \Aut (G_i)$ that
reverses $P_i$. Moreover, the middle two vertices of the path $P_i$ have a
common neighbour that is fixed by $\tau_i$.
\end{lemma}

\begin{proof}
The statement clearly holds for $G_0$ and $P_0$; see Figure 17(a). If we apply
the $\mathcal{D}$-extension to $(G_i, P_i)$, the involution $\tau_i$ extends to
$\tau_{i + 1}$ by $x_4 \leftrightarrow x_5$, $x_0 \leftrightarrow x_3$ and $x_1
\leftrightarrow x_2$; see Figure 16(b). The middle two vertices of $P_{i + 1}$
are the same as those of $P_i$, therefore the fixed common neighbour remains the
same. Likewise, if we apply the $\mathcal{M}$-extension, the involution $\tau_i$
extends to $\tau_{i + 1}$ by $x_0 \leftrightarrow x_3$, $x_1 \leftrightarrow
x_2$ while fixing $x_4$ and $x_5$; see Figure 16(c). The middle two vertices of
$P_{i + 1}$ are $x_1$ and $x_2$, with common neighbour $x_4$.
\end{proof}

\noindent For an automorphism $\alpha \in \Aut(G)$ and a vector
$\mathbf{x}\colon V(G)\to\mathbb{R}$, we write
$(\alpha\mathbf{x})(v)=\mathbf{x}(\alpha(v))$. A vector $\x \colon V(G) \to
\mathbb{R}$ is symmetric with respect to $\tau$ if $\x(u) = (\tau \x)(u)$ for
every $u \in V(G)$. Likewise, $\x$ is antisymmetric if $\x(u) = - (\tau \x)(u)$
for every $u \in V(G)$. We will use this fact from linear algebra.

\begin{proposition}
\label{prop:kernelDecompose}
Let $G$ be a graph with an involution $\tau \in \Aut(G)$. Every $\x \in \ker
(G)$ can be expressed uniquely as $\x = \x^+ + \x^-$, where $\x^+, \x^- \in \ker
(G)$ are such that $\x^+$ is symmetric and $\x^-$ is antisymmetric. Moreover,
$\ker G = N^+(G) \oplus N^-(G)$, where $N^+(G)$ and $N^-(G)$ are the symmetric
and antisymmetric parts of $\ker G$, respectively.
\end{proposition}

Note that $\x^-(u) = 0$ if $u$ is fixed by $\tau$. The spaces $N^+(G)$ and
$N^-(G)$ depend on $\tau$, so strictly speaking, we should write $N^+(G, \tau)$
and $N^-(G, \tau)$, respectively. As the involution $\tau$ is clear from the
context, we omit it.

\begin{lemma}
\label{lem:symmetricKernel}
For every $i \geq 0$, $N^-(G_i)$ is trivial and $\ker G_i = N^+(G_i)$.
\end{lemma}

\begin{proof}
The proof goes by induction. For the seed graph $G_0$ (in Figure 17) this can be
checked by direct computation. Now suppose that $(G_{i + 1}, P_{i + 1}) =
\mathcal{D}(G_i, P_i)$. Let $\x \in N^-(G_{i + 1})$. Let $w$ denote the common
neighbour of $v_1$ and $v_2$ in $G_{i+1}$. As $w$ is fixed by $\tau_{i + 1}$,
$\x(w) = 0$. Denote $\x(v_1)$ by $\alpha$. By using the local condition at
vertices $x_0, x_1, x_4$ and $v_1$, we obtain $\x(v_3) = \x(x_3) = \x(x_1) =
\alpha$ and $\x(v_0) = \x(x_0) = \x(x_2) = \x(v_2) = -\alpha$ as well as
$\x(x_4) = \x(x_5) = 0$. Let $\widehat{\x}$ be $\x$ restricted to graph $G_i$.
Observe that
\[
(\mathbf{A}(G_i) \widehat{\x}) (u) = \begin{cases}
2 \alpha, & \text{if } u = v_0 \text{ or } u = v_2; \\
-2 \alpha, & \text{if } u = v_1 \text{ or } u = v_3; \\
 0, & \text{otherwise}.
\end{cases}
\]
Define $\y \colon V(G_i) \to \mathbb{R}$ by $\y(v_1) = -2\alpha$, $\y(v_2) =
2\alpha$ and $\y(u) = 0$ for $u \in V(G_i) \setminus \{ v_1, v_2 \}$. It is easy
to check that $\mathbf{A}(G_i) (\widehat{\x} + \y) = \mathbf{0}$. Observe that
$\widehat{\x} + \y \in N^-(G_i)$ as $\x$ and $\y$ are both antisymmetric. By the
induction hypothesis, $N^-(G_i)$ is trivial, therefore $\widehat{\x} = -\y$.
From $\widehat{\x}(v_0) = \x(v_0) = - \alpha$ and $\y(v_0) = 0$ it follows that
$\alpha = 0$. This implies that $\x = \mathbf{0}$, so $N^-(G_{i+1})$ is trivial.

The case $(G_{i + 1}, P_{i + 1}) = \mathcal{M}(G_i, P_i)$ goes by analogy.

From Proposition~\ref{prop:kernelDecompose} we now immediately obtain $\ker G_i
= N^+(G_i)$.
\end{proof}

\noindent By Lemma~\ref{lem:symmetricKernel}, every kernel eigenvector of $G_i$
for all $i \geq 0$ from now on can be assumed to be symmetric.
 
\begin{lemma}
\label{lem:imperfectSolution}
Suppose that $\eta(G_i) = 1$ and that $\ker G_i$ is spanned by a vector $\s$
such that $\s(v_0) = \s(v_1) = \s(v_2) = \s(v_3) = 1$. Let $\z_0 \colon V(G_i)
\to \mathbb{R}$ be defined as $\z_0(v_0) = \z_0(v_3) = 1$, $\z_0(v_1) =
\z_0(v_2) = -1$, and $\z_0(u) = 0$ for $u \in V(G_i) \setminus \{v_0, v_1, v_2,
v_3\}$. Then there exists a symmetric $\y_0$ such that $\mathbf{A}(G_i) \y_0 =
\z_0$. Let $\rho(G_i, P_i) \coloneqq \y_0(v_0) - \y_0(v_1)$. Then $\y(v_0) -
\y(v_1) = \rho(G_i, P_i)$ for all $\y$ that satisfy $\mathbf{A}(G_i) \y = \z_0$.
\end{lemma}

\begin{proof}
It is easy to see that $\langle \s , \z_0 \rangle = 0$. Since $\eta(G_i) = 1$,
$\z_0$ is orthogonal to the entire $\ker \mathbf{A}(G_i)$. Recall that $\im
\mathbf{A}(G_i)$ is the orthogonal complement of $\ker \mathbf{A}(G_i)$. Thus,
there exists $\y_1$ such that $\mathbf{A}(G_i) \y_1 = \z_0$. It is easy to see
that $\y_0 = \frac{1}{2} \y_1 + \frac{1}{2} \tau_i \y_1$ is a symmetric vector
that satisfies $\mathbf{A}(G_i) \y_0 = \z_0$.

Now, let $\y$ be another vector that satisfies $\mathbf{A}(G_i) \y = \z_0$. Then
$\mathbf{A}(G_i)(\y - \y_0) = \mathbf{A}(G_i) \y - \mathbf{A}(G_i) \y_0 = 0$, so
$\y - \y_0 \in \ker \mathbf{A}(G_i)$. In other words, $\y - \y_0 = \mu \s$ for
some scalar $\mu$. In particular $\y(v_0) - \y_0(v_0) = \y(v_1) - \y_0(v_1) =
\mu$, and the statement follows.
\end{proof}

\begin{lemma}
\label{lem:symmetricKernelDog}
Let $(G_i,P_i)$ be as in Definition~\ref{def:manDogFamily}, and suppose that
$(G_{i+1},P_{i+1})=\mathcal{D}(G_i,P_i)$. Suppose that $\rho(G_i, P_i) \neq -1$,
that $\eta(G_i) = 1$ and that $\ker G_i$ is spanned by a vector $\s$ such that
$\s(v_0) = \s(v_1) = \s(v_2) = \s(v_3) = 1$. Then $\eta(G_{i+1}) = 1$ and $\s$
extends to $\widetilde{\s} \in \ker G_{i+1}$ by $\widetilde{\s}(x_0) =
\widetilde{\s}(x_1) = \widetilde{\s}(x_2) = \widetilde{\s}(x_3) = 1$ and
$\widetilde{\s}(x_4) = \widetilde{\s}(x_5) = - 2$. Moreover, $\rho(G_{i+1},
P_{i+1}) = -1$.
\end{lemma}

\begin{proof}
Let $\x \in \ker G_{i+1}$. Let $w$ denote the common neighbour of $v_1$ and
$v_2$ in $G_i$. Denote $\x(v_0)$ by $\alpha$ and $\x(v_1)$ by $\beta$. By
Lemma~\ref{lem:symmetricKernel}, $\x$ is symmetric. Using the local conditions
at the vertices $x_0, x_1, x_4$ and $v_1$, we obtain $\x(v_3) = \x(x_0) =
\x(x_3) = \alpha$ and $\x(v_2) = \x(x_1) = \x(x_2) = \beta$ as well as $\x(x_4)
= \x(x_5) = -\alpha - \beta$ and $\x(w) = - 2\beta$. Let $\widehat{\x}$ be the
restriction of $\x$ to graph $G_i$. Observe that
\begin{equation}
\label{eq:restrictionDog}
(\mathbf{A}(G_i) \widehat{\x}) (u) = \begin{cases}
\beta - \alpha, & \text{if } u = v_0 \text{ or } u = v_3; \\
\alpha - \beta, & \text{if } u = v_1 \text{ or } u = v_2; \\
 0, & \text{otherwise}.
\end{cases}
\end{equation}
Note that the right-hand side of \eqref{eq:restrictionDog} is $(\beta -
\alpha)\z_0$, where $\z_0$ is as defined in Lemma~\ref{lem:imperfectSolution}.
From Lemma~\ref{lem:imperfectSolution} we obtain $\widehat{\x}(v_0) -
\widehat{\x}(v_1) = (\beta - \alpha)\rho(G_i, P_i)$. On the other hand,
$\widehat{\x}(v_0) - \widehat{\x}(v_1) = \x(v_0) - \x(v_1) = \alpha - \beta$.
Therefore, $(\beta - \alpha)(\rho(G_i, P_i) + 1) = 0$. By assumption $\rho(G_i,
P_i) \neq -1$, which implies that $\alpha = \beta$, which in turn implies
$\widehat{\x} \in \ker \mathbf{A}(G_i)$ and thus $\widehat{\x} = \beta \s$.
Thus, the entries $\x(x_0) = \x(x_1) = \x(x_2) = \x(x_3) = \beta$ and $\x(x_4) =
\x(x_5) = -2\beta$ are uniquely determined. It is easy to check that
$\widetilde{\s} \in \ker G_{i+1}$ and therefore $\eta(G_{i+1}) = 1$.

By Lemma~\ref{lem:imperfectSolution}, there exists a symmetric vector $\y$
satisfying
\[
(\mathbf{A}(G_{i+1}) \y)(u) = \begin{cases}
1, & \text{if } u = x_1 \text{ or } u = x_2; \\
-1, & \text{if } u = v_1 \text{ or } u = v_2; \\
 0, & \text{otherwise}.
\end{cases}
\]
In particular, we have $\y(x_0) + \y(x_4) + \y(v_1) = 1$ and $\y(x_0) + \y(x_1)
+ \y(x_5) = 0$. As $\y$ is symmetric, $\y(x_5) = \y(x_4)$. From here, we obtain
$\y(x_1) - \y(v_1) = -1$. Recall that $P_{i+1}=x_1 v_1 v_2 x_2$. Thus,
$\rho(G_{i+1}, P_{i+1}) = -1$.
\end{proof}

\begin{lemma}
\label{lem:symmetricKernelMan}
Let $(G_i,P_i)$ be as in Definition~\ref{def:manDogFamily}, and suppose that
$(G_{i+1},P_{i+1})=\mathcal{M}(G_i,P_i)$. Suppose that $\rho(G_i, P_i) \neq -4$,
that $\eta(G_i) = 1$ and that $\ker G_i$ is spanned by a vector $\s$ such that
$\s(v_0) = \s(v_1) = \s(v_2) = \s(v_3) = 1$. Then $\eta(G_{i+1}) = 1$ and $\s$
extends to $\widetilde{\s} \in \ker G_{i + 1}$ by $\widetilde{\s}(x_0) =
\widetilde{\s}(x_1) = \widetilde{\s}(x_2) = \widetilde{\s}(x_3) = 1$ and
$\widetilde{\s}(x_4) = \widetilde{\s}(x_5) = - 2$. Moreover, $\rho(G_{i+1},
P_{i+1}) = \frac{\rho(G_i,P_i)}{\rho(G_i,P_i)+4}$.
\end{lemma}

\begin{proof}
Let $\x \in \ker G_{i+1}$. Let $w$ denote the common neighbour of $v_1$ and
$v_2$ in $G_i$. Denote $\x(v_0)$ by $\alpha$ and $\x(v_1)$ by $\beta$. By
Lemma~\ref{lem:symmetricKernel}, $\x$ is symmetric. Using the local conditions
at the vertices $x_0, x_1, x_4, x_5$ and $v_1$, we obtain $\x(x_0) = \x(x_3) =
\frac{\alpha + 3\beta}{4}$, $\x(x_1) = \x(x_2) = \frac{\alpha + \beta}{2}$,
$\x(x_4) = \frac{-\alpha - 3\beta}{2}$, $\x(x_5) = -\alpha - \beta$ and $\x(w) =
\frac{-3\alpha - 5\beta}{4}$. Let $\widehat{\x}$ be the restriction of $\x$ to
graph $G_i$. Observe that
\begin{equation}
\label{eq:restrictionMan}
(\mathbf{A}(G_i) \widehat{\x}) (u) = \begin{cases}
\frac{\beta - \alpha}{4}, & \text{if } u = v_0 \text{ or } u = v_3; \\
\frac{\alpha - \beta}{4}, & \text{if } u = v_1 \text{ or } u = v_2; \\
 0, & \text{otherwise}.
\end{cases}
\end{equation}
Note that the right hand side of \eqref{eq:restrictionMan} is $\frac{1}{4}(\beta
- \alpha)\z_0$, where $\z_0$ is as defined in Lemma~\ref{lem:imperfectSolution}.
From Lemma~\ref{lem:imperfectSolution} we obtain $\widehat{\x}(v_0) -
\widehat{\x}(v_1) = \frac{1}{4}(\beta - \alpha)\rho(G_i, P_i)$. On the other
hand, $\widehat{\x}(v_0) - \widehat{\x}(v_1) = \x(v_0) - \x(v_1) = \alpha -
\beta$. Therefore, $(\beta - \alpha)(\frac{1}{4}\rho(G_i, P_i) + 1) = 0$. By
assumption $\rho(G_i, P_i) \neq -4$, which implies that $\alpha = \beta$, which
in turn implies $\widehat{\x} \in \ker \mathbf{A}(G_i)$ and thus $\widehat{\x} =
\beta \s$. Thus, the entries $\x(x_0) = \x(x_1) = \x(x_2) = \x(x_3) = \beta$ and
$\x(x_4) = \x(x_5) = -2\beta$ are uniquely determined. It is easy to check that
$\widetilde{\s} \in \ker G_{i+1}$ and therefore $\eta(G_{i+1}) = 1$.

It remains to determine $\rho(G_{i+1}, P_{i+1}) $. By
Lemma~\ref{lem:imperfectSolution}, there exists a symmetric vector $\y$
satisfying
\[
(\mathbf{A}(G_{i+1}) \y)(u) = \begin{cases}
1, & \text{if } u = v_1 \text{ or } u = v_2; \\
-1, & \text{if } u = x_1 \text{ or } u = x_2; \\
 0, & \text{otherwise}.
\end{cases}
\]
By using the local conditions at vertices $x_0, x_1, x_4$ and $x_5$, we obtain
$\y(v_0) + \y(x_5) + \y(v_1) = 0$, $\y(v_1) + \y(x_1) + \y(x_4) = -1$, $2\y(x_1)
+ \y(x_5) = 0$ and $2 \y(x_0) + \y(x_4) = 0$. Hence, $\y(v_0) = 2\y(x_1) -
\y(v_1)$ and $\y(x_0) = \frac{1}{2}(1 + \y(v_1) + \y(x_1))$. Recall that
$P_{i+1}= v_1 x_1 x_2 v_2$, and note that $\rho(G_{i+1}, P_{i+1}) = \y(v_1) -
\y(x_1)$. Let $\widehat{\y}$ be the restriction of $\y$ to $G_i$. Observe that
\[
(\mathbf{A}(G_i) \widehat{\y})(u) = \begin{cases}
\frac{1}{2}(\rho(G_{i+1}, P_{i+1}) - 1),
  & \text{if } u = v_0 \text{ or } u = v_3; \\
-\frac{1}{2}(\rho(G_{i+1}, P_{i+1}) - 1),
  & \text{if } u = v_1 \text{ or } u = v_2; \\
 0, & \text{otherwise}.
\end{cases}
\]
By another application of Lemma~\ref{lem:imperfectSolution}, we have $\y(v_0) -
\y(v_1) = \frac{1}{2}(\rho(G_{i+1}, P_{i+1}) - 1) \rho(G_i, P_i)$. Substituting
$2\y(x_1) - \y(v_1)$ for $\y(v_0)$, we obtain $-2 \rho(G_{i+1}, P_{i+1}) =
\frac{1}{2}(\rho(G_{i+1}, P_{i+1}) - 1) \rho(G_i, P_i)$. Hence, $\rho(G_{i+1},
P_{i+1}) = \frac{\rho(G_i,P_i)}{\rho(G_i,P_i)+4}$. Since $\rho(G_i, P_i) \neq
-4$, we are allowed to divide by $\rho(G_i,P_i)+4$.
\end{proof}

\begin{proof}[Proof of Theorem~\ref{thm:polyhedral}]
We first construct one polyhedral nut graph of bipartite edge frustration $k$
for every $k \geq 2$. The cases $k = 2$ and $k = 3$ are provided by
Table~\ref{tbl:3}. Let $G_0$ be the graph shown in
Figure~\ref{fig:polyhedralSeed}(a), with the distinguished path $P_0$ indicated.
It is straightforward to verify that $G_0$ is planar and $3$-connected.
Moreover, $\eta(G_0) = 1$ and $\varphi(G_0) = 4$. The entries of a kernel
eigenvector of $G_0$ are displayed in Figure~\ref{fig:polyhedralSeed}(a); in
particular, the eigenvector is full, and hence $G_0$ is a nut graph.

Starting from $(G_0,P_0)$, construct the sequence $(G_0,P_0), (G_1,P_1),
(G_2,P_2), \ldots$ as in Definition~\ref{def:manDogFamily}. By
Lemmas~\ref{lem:symmetricKernelDog} and~\ref{lem:symmetricKernelMan}, each graph
$G_i$ has nullity one and its kernel eigenvector is full, i.e., every $G_i$ is a
nut graph. Observe that $\rho(G_0,P_0) = \frac{1}{3}$ (a vector $\y_0$ of
Lemma~\ref{lem:imperfectSolution} is shown in
Figure~\ref{fig:polyhedralSeed}(b)). Moreover, $\rho(G_1,P_1) = \frac{1}{13}$,
$\rho(G_i,P_i) = -1$ for all even $i \geq 2$ and $\rho(G_i,P_i) = -\frac{1}{3}$
for all odd $i \geq 3$.

Next, we verify that the graphs $G_i$ are polyhedral. Planarity follows
inductively from the construction. Indeed, at every step the path $P_i$ lies on
the boundary of a face into which the corresponding $\mathcal M$- or $\mathcal
D$-gadget can be inserted without introducing an edge crossing; see
Figures~\ref{fig:polyhedralSeed}(c) and~\ref{fig:polyhedralSeed}(d) for
illustration. Thus $G_{i+1}$ is planar whenever $G_i$ is planar. Similarly,
$3$-connectivity is preserved by both operations $\mathcal M$ and $\mathcal D$.
This can be verified directly using Menger's Theorem.

By Lemma~\ref{lem:hadlock}, $\varphi(G_i) = i + 4$. More precisely, $G_0^*$
contains $8$ vertices of odd degree, and there is a matching $E_0$ of size $4$
that covers these $8$ vertices. Note that $G_{i+1}^*$ has $2$ more vertices of
odd degree than $G_i^*$. We can obtain matching $E_{i + 1}$ (of size $i + 5$) by
a local modification of $E_i$.

The family $G_0, G_1, G_2, \ldots{}$, together with the examples for $k=2$ and
$k=3$, gives one polyhedral nut graph with $\varphi(G) = k$ for every $k \geq
2$. For each such graph, repeated application of
Lemma~\ref{lem:fowlerFrustration} produces infinitely many pairwise
non-isomorphic polyhedral nut graphs with a given value $k$.
\end{proof}

\section{New constructions for extending nut graphs}
\label{sec:newConstruct}

\subsection{The cut construction}

In this section, we take a different direction and show how considerations from
bipartivity lead to new general constructions for nut graphs. Recall the
subdivision and the bridge constructions from
Proposition~\ref{prop:bridgeSubdivPreserve}. They both make nut graphs from a
starting nut graph. We will generalise the bridge construction, but first we
recall a theorem from~\cite{Gutman1998}:

\begin{theorem}[\cite{Gutman1998}]
\label{thm:GutmanTheorem}
Let $G$ be a nut graph with kernel eigenvector $\x$ and let $e = uv \in E(G)$ be
a bridge. Let $G - e = G_u \sqcup G_v$, where $G_u$ (resp.\ $G_v$) is the
connected component that contains vertex $u$ (resp.\ $v$). Let $\mathcal{B}(G,
e)$ be the graph obtained from $G$ by subdividing the edge $e$ with two new
vertices $w_1$ and $w_2$, such that $\{uw_1, w_1w_2, w_2v\} \subseteq
E(\mathcal{B}(G, e))$. Then $\mathcal{B}(G, e)$ is a nut graph with kernel
eigenvector $\x'$ defined as
\begin{equation}
\x'(t) = 
\begin{cases}
\phantom{-}\x(t) & \text{if } t \in V(G_u), \\
-\x(t) & \text{if } t \in V(G_v), \\
\phantom{-}\x(v) & \text{if } t = w_1, \\
-\x(u) & \text{if } t = w_2.
\end{cases}
\end{equation}
\end{theorem}

\noindent This can be generalised to double subdivision of the edges comprising
any cut set.

\begin{theorem}[The cut construction]
\label{thm:CutConstruction}
Let $G$ be a graph and let $A \subseteq V(G)$. Let $\mathcal{C}(G, A)$ be the
graph obtained from $G$ by subdividing with two vertices every edge from
$\delta(A)$. Then $\mathcal{C}(G, A)$ is a nut graph if and only if $G$ is a nut
graph.
\end{theorem}

\begin{figure}[!htbp]
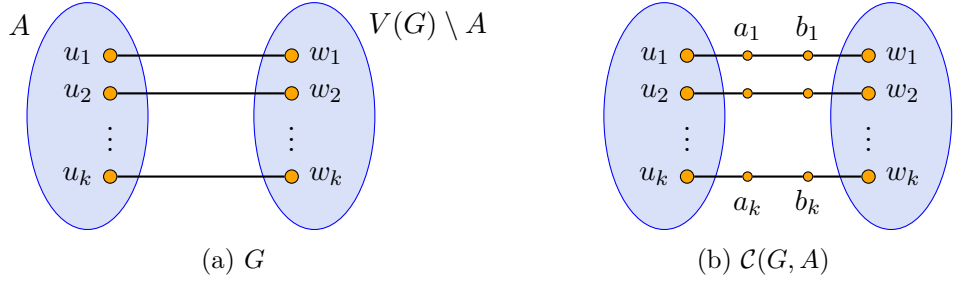

\centering
\subcaptionbox{$G$} { 
\input{img/cut_construction_panela.tikz}
}
\qquad
\subcaptionbox{$\mathcal{C}(G, A)$} { 
\input{img/cut_construction_panelb.tikz}
}
\caption{The cut construction.}
\label{fig:theorem3}
\end{figure}

\begin{proof}
Let the edges of $\delta(A)$ be denoted $e_1 = u_1 w_1, e_2 = u_2 w_2, \ldots,
e_k = u_k w_k$, where $u_i \in A$ and $w_i \in V(G) \setminus A$. Let
$V(\mathcal{C}(G, A)) = V(G) \cup \{a_i, b_i \mid i = 1,\ldots, k \}$ and
$E(\mathcal{C}(G, A)) = \omega(A) \cup \{ u_i a_i, a_i b_i, b_i w_i \mid i = 1,
\ldots, k \}$; see the schematic Figure~\ref{fig:theorem3}. Note that vertices
$u_1, \ldots, u_k$ (and $w_1, \ldots, w_k$) are not necessarily distinct.

First, we prove that $\eta(\mathcal{C}(G, A)) = \eta(G)$. Let $\x \in \ker G$.
Let us define $\widetilde{\x}\colon V(\mathcal{C}(G, A)) \to \mathbb{R}$ by
\begin{equation}
\widetilde{\x}(t) = 
\begin{cases}
\phantom{-}\x(t) & \text{if } t \in A, \\
-\x(t) & \text{if } t \in V(G) \setminus A, \\
\phantom{-}\x(w_i) & \text{if } t = a_i, \\
-\x(u_i) & \text{if } t = b_i.
\end{cases}
\label{eq:definitionOfWidetildeX}
\end{equation}
By checking local conditions for vertices of $\mathcal{C}(G, A)$ it is easy to
see that $\widetilde{\x} \in \ker \mathcal{C}(G, A)$. From basis vectors
$\x^{(1)}, \x^{(2)}, \ldots , \x^{(\eta)}$ that span $\ker G$, we obtain
$\widetilde{\x}^{(1)}, \widetilde{\x}^{(2)}, \ldots , \widetilde{\x}^{(\eta)}$.
As vectors $\widetilde{\x}^{(1)}, \widetilde{\x}^{(2)}, \ldots ,
\widetilde{\x}^{(\eta)}$ are linearly independent, it follows that
$\eta(\mathcal{C}(G, A)) \geq \eta(G)$.

Now, suppose that $\y \in \ker\mathcal{C}(G, A)$. From the local condition at
vertex $a_i$ it follows that $\y(u_i) = -\y(b_i)$ for every $i = 1, \ldots, k$.
Similarly, from the local condition at $b_i$ it follows that $\y(w_i) =
-\y(a_i)$ for every $i = 1, \ldots, k$. Let us define $\widehat{\y}\colon V(G)
\to \mathbb{R}$ by
\begin{equation}
\widehat{\y}(t) = 
\begin{cases}
\phantom{-}\y(t) & \text{if } t \in A, \\
-\y(t) & \text{if } t \in V(G) \setminus A. 
\end{cases}
\label{eq:definitionOfWidethatY}
\end{equation}
By checking local conditions for all vertices of $G$ it is easy to see that
$\widehat{\y} \in \ker G$. Similarly as before, we consider basis vectors
$\y^{(1)}, \y^{(2)}, \ldots , \y^{(\mu)}$ that span $\ker \mathcal{C}(G, A)$.
From these vectors we obtain $\widehat{\y}^{(1)}, \widehat{\y}^{(2)}, \ldots ,
\widehat{\y}^{(\mu)}$ that are linearly independent, and hence
$\eta(\mathcal{C}(G, A)) \leq \eta(G)$. Therefore, $\eta(\mathcal{C}(G, A)) =
\eta(G)$.

From \eqref{eq:definitionOfWidetildeX} it is easy to see that $\widetilde{\x}$
contains no zero entries if $\x$ contained none. Likewise, from
\eqref{eq:definitionOfWidethatY}, $\widehat{\y}$ contains no zero entries if
$\y$ contained none. Hence, the theorem follows.
\end{proof}

Note that in Theorem~\ref{thm:CutConstruction}, `nut graph' can be replaced
everywhere by `core graph'. As a bridge is a cut set of size $1$,
Theorem~\ref{thm:CutConstruction} subsumes Theorem~\ref{thm:GutmanTheorem},
which was proved in a different way in \cite{Gutman1998}.

\subsection{The frustration construction}

It is possible to define another construction by choosing to subdivide the edges
of the frustration set instead of those of the cut set. The notation
$\mathcal{F}(G, A)$ which we use here is not to be confused with $F(G,v)$ as
used in {\S}~\ref{subsec:fowler}.

\begin{theorem}[The frustration construction]
\label{thm:FrustrationConstruction}
Let $G$ be a graph and let $A \subseteq V(G)$. Let $\mathcal{F}(G, A)$ be the
graph obtained from $G$ by subdividing with two vertices every edge from
$\omega(A)$. Then $\mathcal{F}(G, A)$ is a nut graph if and only if $G$ is a nut
graph.
\end{theorem}

\begin{figure}[!htbp]
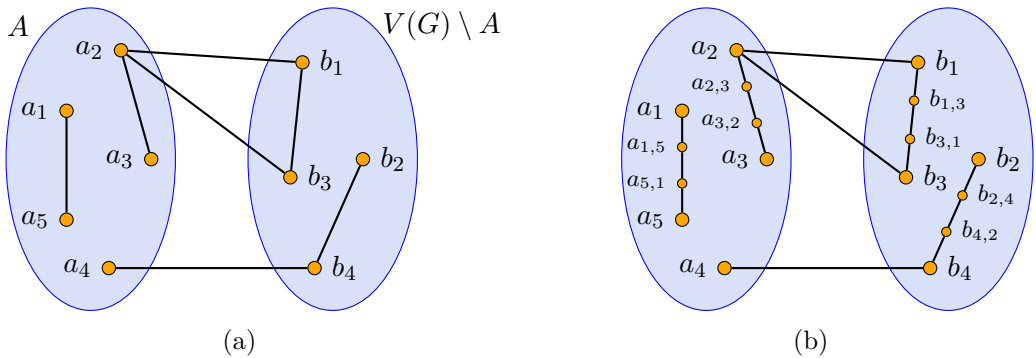

\centering
\subcaptionbox{} {
\input{img/fig16_panela.tikz}
}
\qquad
\subcaptionbox{} { 
\input{img/fig16_panelb.tikz}
}
\caption{The frustration construction.}
\label{fig:theorem4}
\end{figure}

\begin{proof}
The proof follows that of Theorem~\ref{thm:CutConstruction} with minor
modifications. Let us label $A = \{ a_1, a_2, \ldots, a_k \}$ and $V(G)
\setminus A = \{ b_1, b_2, \ldots, b_{n-k} \}$, where $n = |V(G)|$. Moreover,
let us label
\[
V(\mathcal{F}(G, A)) = V(G) 
\cup \{ a_{i j} \mid 1 \leq i, j \leq k, a_ia_j \in E(G) \} 
\cup \{ b_{i j} \mid 1 \leq i, j \leq n - k, b_ib_j \in E(G) \}.
\]
Note that the two new vertices that were obtained by subdividing the edge $a_i
a_j$ are labelled $a_{i j}$ and $a_{j i}$; see the schematic
Figure~\ref{fig:theorem4}. The vertex $a_{i j}$ is closer to $a_i$ than to
$a_j$, and the analogous convention is used for $b_{i j}$. The set of edges of
$\mathcal{F}(G, A)$ is thus
\begin{align*}
E(\mathcal{F}(G, A)) & = \{ a_ib_j \mid a_ib_j \in E(G) \} \\
 & \qquad \cup \{ a_i a_{i j}, a_{i j} a_{j i}, a_{j i} a_j
    \mid a_i a_j \in E(G) \}
    \cup \{ b_i b_{i j}, b_{i j} b_{j i}, b_{j i} b_j \mid b_i b_j \in E(G) \}.
\end{align*}
First, we prove that $\eta(\mathcal{F}(G, A)) = \eta(G)$. Let $\x \in \ker G$.
Let us define $\widetilde{\x}\colon V(\mathcal{F}(G, A)) \to \mathbb{R}$ by
\begin{equation}
\widetilde{\x}(t) = 
\begin{cases}
\phantom{-}\x(t) & \text{if } t \in A, \\
-\x(t) & \text{if } t \in V(G) \setminus A, \\
-\x(a_j) & \text{if } t = a_{ij}, \\
\phantom{-}\x(b_j) & \text{if } t = b_{ij}.
\end{cases}
\label{eq:definitionOfWidetildeX2}
\end{equation}
By checking the local conditions for vertices of $\mathcal{F}(G, A)$ it is easy
to see that $\widetilde{\x} \in \ker (\mathcal{F}(G, A))$. Using the same line
of argument as in the proof of Theorem~\ref{thm:CutConstruction}, one can check
that $\eta(\mathcal{F}(G, A)) \geq \eta(G)$.

Similarly, suppose that $\y \in \ker\mathcal{F}(G, A)$. From the local condition
at vertices $a_{ij}$ and $a_{ji}$ it follows that $\y(a_i) = -\y(a_{ji})$ and
$\y(a_j) = -\y(a_{ij})$. Likewise, from the local condition at vertices $b_{ij}$
and $b_{ji}$, we have $\y(b_i) = -\y(b_{ji})$ and $\y(b_j) = -\y(b_{ij})$. Let
us define $\widehat{\y}\colon V(G) \to \mathbb{R}$ by
\begin{equation}
\widehat{\y}(t) = 
\begin{cases}
\phantom{-}\y(t) & \text{if } t \in A, \\
-\y(t) & \text{if } t \in V(G) \setminus A. 
\end{cases}
\label{eq:definitionOfWidethatY2}
\end{equation}
By checking the local conditions for vertices of $G$ it is easy to see that
$\widehat{\y} \in \ker G$, and to argue that $\eta(\mathcal{F}(G, A)) \leq
\eta(G)$. Therefore, $\eta(\mathcal{F}(G, A)) = \eta(G)$.

From \eqref{eq:definitionOfWidetildeX2} it is easy to see that $\widetilde{\x}$
contains no zero entries if $\x$ contained none. Likewise, from
\eqref{eq:definitionOfWidethatY2}, $\widehat{\y}$ contains no zero entries if
$\y$ contained none. Hence, the theorem follows.
\end{proof}

Note that in Theorem~\ref{thm:FrustrationConstruction}, `nut graph' can be
replaced everywhere by `core graph'. Theorems~\ref{thm:CutConstruction}
and~\ref{thm:FrustrationConstruction} are complementary in the following sense.
We choose a cut $A$ in graph $G$; the cut $A$ is not necessarily maximum. Then
we can subdivide either the cut set or the frustration set. Note that this
frustration set does not necessarily correspond to a max-cut and thus does not
necessarily realise bipartite edge frustration. The edges that are subdivided in
one theorem are left intact in the other, and vice versa. The combination of
Theorems~\ref{thm:CutConstruction} and~\ref{thm:FrustrationConstruction} can
yield a veritable cornucopia of nut graphs. In particular, if we take $A = V(G)$
in Theorem~\ref{thm:FrustrationConstruction}, we have the easy but still
striking observation that double subdivision of all edges of a nut graph leads
to another nut graph. We could call this construction \emph{total double
subdivision}.

\begin{corollary}[Total double subdivision]
Let $G$ be a graph and let $\mathcal{O}(G)$ be the graph obtained from $G$ by
subdividing with two vertices every edge of $G$. Then $\mathcal{O}(G)$ is a nut
graph if and only if $G$ is a nut graph.
\end{corollary}

\noindent It is easy to see that the cut, frustration and total double
subdivision constructions all preserve the bipartite edge frustration.

\begin{proposition}
Let $G$ be a graph and let $A \subseteq V(G)$. Then
\[
\varphi(\mathcal{F}(G, A)) = \varphi(\mathcal{C}(G, A)) = 
\varphi(\mathcal{O}(G)) = \varphi(G).
\]
\end{proposition}

The crucial observation is the following. Suppose that a graph $G$ contains a
path $P = v_0 v_1 v_2 \ldots v_k$, where $\deg(v_i) = 2$ for all $0 < i < k$.
Let $c \colon V(G) \to \{ b, w \}$ be a colouring that realises the bipartite
edge frustration. Then path $P$ contains at most one frustrated edge in
colouring $c$. To see this, observe that changing the colour of vertex $v_i$
reverses the status of each of its incident edges ($v_{i-1}v_i$ and
$v_iv_{i+1}$) from frustrated to non-frustrated (or vice versa).

\section{Future work}
\label{Zukunft}

In {\S}\ref{sec:5dot4}, we studied regular nut graphs with prescribed bipartite
edge frustration. We believe that for every $d \geq 3$ and `feasible' $k \geq 2$
there exists a $d$-regular nut graph $G$ with $\varphi(G) = k$. Here, feasible
means that $k$ adheres to Proposition~\ref{prop:explainMissingFrust}. More
precisely, we conjecture the following.

\begin{conjecture}
For $d \geq 3$ the following statements hold.
\begin{enumerate}[label=(\roman*)]
\item If $d$ is odd, then for all $k \geq d$ and for all even $k \geq 2$ there
exists a $d$-regular nut graph $G$ with $\varphi(G) = k$. Moreover, there exists
no $d$-regular nut graph $G$ with odd $\varphi(G)$ where $\varphi(G) < d$.
\item If $d \equiv 0 \pmod{4}$, then for every even $k \geq 2$ there exists a
$d$-regular nut graph $G$ with $\varphi(G) = k$. Moreover, there exists no
$d$-regular nut graph $G$ with odd $\varphi(G)$.
\item If $d \equiv 2 \pmod{4}$, then for every even $k \geq 2$ and for every odd
$k \geq \frac{d}{2}$ there exists a $d$-regular nut graph $G$ with $\varphi(G) =
k$. Moreover, $|V(G)| \equiv k \pmod{2}$, and there exists no $d$-regular nut
graph $G$ with odd $\varphi(G)$ where $\varphi(G) < \frac{d}{2}$.
\end{enumerate}
\end{conjecture}

\noindent The cases where $3 \leq d \leq 8$ are settled by
Corollary~\ref{thm:cubicAnyFrus}, Theorems~\ref{thm:quarticAnyFrus}, and
\ref{thm:quinticAnyFrus}--\ref{thm:octicAnyFrus}. The proof technique used for
Theorems~\ref{thm:quinticAnyFrus}--\ref{thm:octicAnyFrus} could be used to add
specific higher values of $d$.

We have shown that for each $k \geq 2$ there exist infinitely many polyhedral
nut graphs with bipartite edge frustration $k$. From Table~\ref{tbl:4} we see
that nut fullerenes exist for $\varphi \in \{6, 7, 8, 10, 12, 14, 16\}$. This
leads us to conjecture the following:

\begin{conjecture}
For every even $\varphi \geq 6$ there exist infinitely many nut fullerenes of
bipartite edge frustration $\varphi$.
\end{conjecture}

\begin{conjecture}
The nut fullerene in Figure~\ref{fig:fullerFru7} is the only nut fullerene with
odd bipartite edge frustration.
\end{conjecture}

\noindent In principle, examples of nut fullerenes for higher values of
$\varphi$ could be found by exhaustive enumeration beyond order $n = 240$. This
would not get us very far, because the number of vertices of the smallest
possible fullerene graph increases quadratically with respect to frustration.

\subsection{Other measures of bipartivity}

Turning to applications, we note that researchers in various disciplines are
interested in assessing bipartivity of networks. The quantity
\begin{equation}
b_{\mathrm A}(G) = 1 - \frac{\varphi(G)}{|E(G)|}
\end{equation} 
has been proposed as a measure of the bipartivity of $G$, and used in the study
of the Ising model in physics and of large relationship networks of various
kinds~\cite{Holme2003}. In this formula, $\varphi(G)$ is often perforce replaced
by some approximate estimate. A radically different approach relies on cheaply
computed properties to provide a proxy for bipartivity \cite{Estrada2005}.
Another measure of bipartivity \cite{Estrada2012} is based on spectral graph
theory, and exploits the fact that spectra of bipartite graphs obey the Pairing
Theorem~\cite{Coulson1}. The \emph{Estrada Index} of a graph $G$ is
\begin{equation}
EE(G) = \sum_{j=1}^{n} e\sp{\lambda_j},
\end{equation}
where $\lambda_j$ for $j = 1, 2, \ldots, n$ are the eigenvalues of the adjacency
matrix $\mathbf{A}(G)$. The Estrada index can be viewed as a weighted sum of odd
and even closed walks in the graph, giving more weight to short walks.
Separation of contributions from even and odd closed walks is achieved by
splitting the exponentials:
\begin{equation}
EE(G) = \sum_{j=1}^{n} \cosh(\lambda_j) + \sum_{j=1}^{n} \sinh(\lambda_j).
\end{equation}
The \emph{spectral bipartivity measure}~\cite{Estrada2012,Estrada2005} is
\begin{equation}
b_{\mathrm S}(G) = \frac
{\sum_{j=1}^{n} \cosh(\lambda_j)}{EE(G)} = 1 - \frac
{\sum_{j=1}^{n} \sinh(\lambda_j)}{EE(G)},
\end{equation}
in direct analogy with $b_{\mathrm A}(G)$. The measure $b_{\mathrm S}(G)$ has
been applied to chemical graphs such as fullerenes~\cite{Doslic2005}. It would
be interesting to check how it performs for nut graphs, and in particular to see
how closely nut graphs approach the bipartite limiting value of $1$. Several
other spectral measures of bipartivity can be constructed~\cite{Aleidan2023},
and it could be informative to see how these perform for nut graphs.

A different kind of bipartivity measure is the \emph{bipartite vertex
frustration}, denoted by $\psi(G)$, i.e.\ the minimum number of vertices of $G$
whose deletion results in a bipartite subgraph. A vertex set whose removal
results in a bipartite graph is often called an \emph{odd-cycle vertex
transversal}~\cite{Fiorini2007,Reed2004}, whether or not it has minimum
cardinality. Thus, $\psi(G)$ is the cardinality of a minimum odd-cycle vertex
transversal of $G$. In fact, for planar graphs and each fixed integer $k$, a
minimum odd-cycle vertex transversal can be found in linear time whenever
$\psi(G) \leq k$; otherwise, the algorithm determines that $\psi(G) >
k$~\cite{Fiorini2008}. Clearly, there is no nut graph $G$ for which $\psi(G) =
0$. However, for every $k \geq 1$ one can easily find a nut graph $G$ with
$\psi(G) = k$. For instance, $\psi(S_1) = 1$, $\psi(S_2) = 2$ and
$\psi(\mathcal{T}_k) = k$ for $k \geq 3$, where $\mathcal{T}_k$ is as defined in
\cite[Proposition 15]{Basic2024}.

\section*{Acknowledgements}
The work of NB is supported in part by the Slovenian Research and Innovation
Agency (Research Programme P1-0294). Part of this research was carried out
while NB was affiliated with the University of Primorska, Slovenia.
PWF thanks the Leverhulme Trust for an Emeritus Fellowship on the theme of
`Modelling molecular currents, conduction and aromaticity' during which this
work was initiated.

\bibliographystyle{amcjoucc}
\bibliography{references}

\clearpage
\pagebreak
\phantomsection
\section*{Appendix: Regular nut graphs of small order}
\addcontentsline{toc}{section}{Appendix: Regular nut graphs of small order}
\label{appendix:a}

This appendix reports stratified calculations of bipartite edge frustration for
$d$-regular nut graphs with degrees $3 \leq d \leq 8$, which are used in the
proofs of Theorems~\ref{thm:quarticAnyFrus} and~\ref{thm:quinticAnyFrus}
to~\ref{thm:octicAnyFrus}.

\begin{table}[!ht]
\centering
\begin{tabular}{|r|r||r|r|r|r|r|r|}
\hline
\multirow{ 2}{*}{$n$} & \multirow{ 2}{*}{$|\mathcal{N}_n^{3}|$} & \multicolumn{6}{c|}{$\varphi$} \\
\cline{3-8}
&  & 2 & 3 & 4 & 5 & 6 & 7 \\
\hline
\hline
12 & 9 & 4 & 5 & & & & \\
18 & 5541 & 784 & 2064 & 2549 & 141 & 3 & \\
20 & 5 & & 1 & 4 & & & \\
22 & 71 & & 1 & 39 & 31 & & \\
24 & 13184574 & 485769 & 2138655 & 7269778 & 3090018 & 199499 & 855 \\
\hline
\end{tabular}%

\caption{Bipartite edge frustration of cubic nut graphs.}
\end{table}

\begin{figure}[!ht]
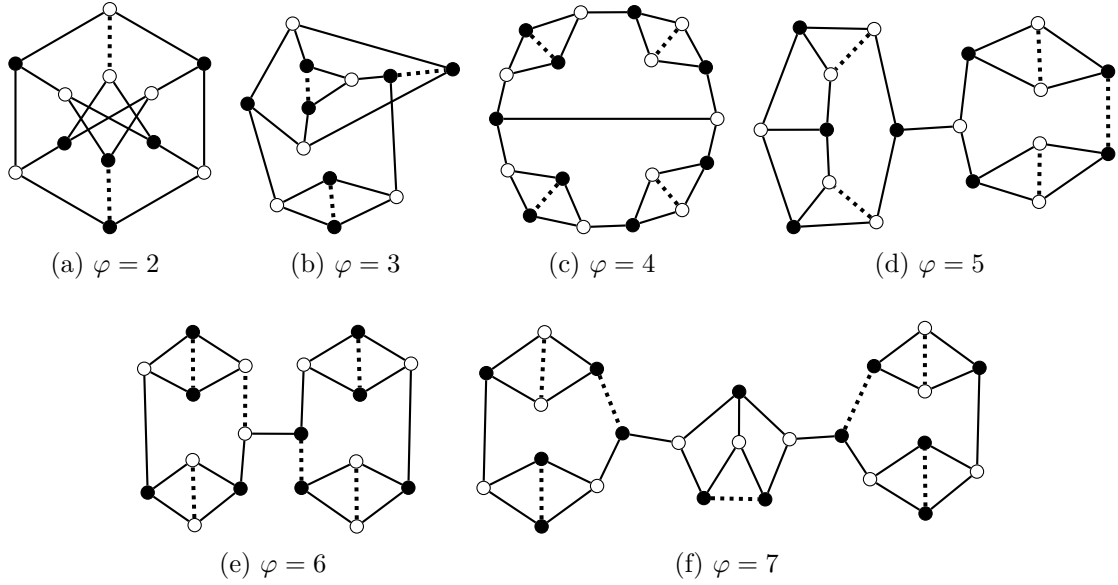

\centering
\subcaptionbox{$\varphi = 2$} { 
\input{img/example_cubic_fru2.tikz}
}
\subcaptionbox{$\varphi = 3$} { 
\input{img/example_cubic_fru3.tikz}
}
\subcaptionbox{$\varphi = 4$} { 
\input{img/example_cubic_fru4.tikz}
}
\subcaptionbox{$\varphi = 5$} { 
\input{img/example_cubic_fru5.tikz}
}

\vspace{\baselineskip}
\subcaptionbox{$\varphi = 6$} { 
\input{img/example_cubic_fru6.tikz}
}
\quad%
\subcaptionbox{$\varphi = 7$} { 
\input{img/example_cubic_fru7.tikz}
}
\caption{Examples of smallest cubic nut graphs with bipartite edge frustration
$2 \leq \varphi \leq 7$.}
\end{figure}

\begin{table}[!ht]
\centering
\begin{tabular}{|r|r||r|r|r|r|r|r|}
\hline
\multirow{ 2}{*}{$n$} & \multirow{ 2}{*}{$|\mathcal{N}_n^{4}|$} & \multicolumn{6}{c|}{$\varphi$} \\
\cline{3-8}
 & & 2 & 4 & 6 & 8 & 10 & 12 \\
\hline
\hline
8 & 1 &  & 1 &  &  & & \\
10 & 12 & 1 & 7 & 4 & & & \\
12 & 269 & 11 & 148 & 110 & & & \\
14 & 15633 & 81 & 3992 & 11257 & 303 & & \\
15 & 1 &  &  & 1 & & & \\
16 & 1348719 & 1460 & 146103 & 1009749 & 191326 &  81 & \\
17 & 4221 & & 155 & 2413 & 1645 & 8 & \\
18 & 75138153 & 39054 & 4357478 & 46414123 & 24231403 & 96081 & 14 \\
\hline
\end{tabular}%

\caption{Bipartite edge frustration of quartic nut graphs.}
\label{tbl:quarticFrustration}
\end{table}

\begin{landscape}

\begin{table}[!p]
\centering
\begin{tabular}{|r|r||r|r|r|r|r|r|r|r|r|r|r|}
\hline
\multirow{2}{*}{$n$} & \multirow{2}{*}{$|\mathcal{N}_n^{5}|$}  & \multicolumn{11}{c|}{$\varphi$} \\
\cline{3-13}
 & & 2 & 4 & 5 & 6 & 7 & 8 & 9 & 10 & 11 & 12 & 13 \\
\hline
\hline
10 & 9 & & & & 5 & 2 & 2 &  &  &  & & \\
12 & 4 & & & & 4 &  &  &  &  &  & & \\
14 & 25 & & & &  &  & 11 & 5 & 9 &  & & \\
16 & 13530 & & & & 59 & 125 & 2213 & 3038 & 7421 & 670 & 4 & \\
18 & 665456900 & 26 & 17922 & 27586 & 2057580 & 4433698 & 67234928 & 111666174 & 344641452 & 119484432 & 15884103 & 8999 \\
\hline
\end{tabular}%

\caption{Bipartite edge frustration of quintic nut graphs.}
\label{tbl:quinticFrustration}
\end{table}

\begin{table}[!p]
\centering
\begin{tabular}{|r|r||r|r|r|r|r|r|r|r|r|r|r|r|r|}
\hline
\multirow{2}{*}{$n$} & \multirow{2}{*}{$|\mathcal{N}_n^{6}|$}  & \multicolumn{12}{c|}{$\varphi$} \\
\cline{3-14}
 & & 3 & 5 & 6 & 7 & 8 & 9 & 10 & 11 & 12 & 13 & 14 & 15 \\
\hline
\hline
12 & 1964 & &  & 2 &  & 180 &  & 1605 &  & 177 & & & \\
13 & 79 & &  &  &  &  & 3 &  & 65 &  & 11 & & \\
14 & 1872  & & &  &  & 1 &  & 184 &  & 1374 &  & 313 & \\
15 & 153665680 & 14 & 2689 & & 286837 & & 9505924 & & 80442331 & & 63146630 & & 281255 \\
\hline
\end{tabular}%

\caption{Bipartite edge frustration of sextic nut graphs.}
\label{tbl:sexticFrustration}
\end{table}

\begin{table}[!ph]
\centering
\begin{tabular}{|r|r||r|r|r|r|r|r|r|r|r|r|}
\hline
\multirow{2}{*}{$n$} & \multirow{2}{*}{$|\mathcal{N}_n^{7}|$} & \multicolumn{10}{c|}{$\varphi$} \\
\cline{3-12}
 & & 8 & 9 & 10 & 11 & 12 & 13 & 14 & 15 & 16 & 17 \\
\hline
\hline
12 & 3 & & & & & 3 & & & & & \\
14 & 5168453 & 51 & 94 & 9963 & 15389 & 417790 & 409059 & 2998502 & 822805 & 494663 & 137 \\
\hline
\end{tabular}%

\caption{Bipartite edge frustration of septic nut graphs.}
\label{tbl:septicFrustration}
\end{table}

\begin{table}[!ht]
\centering
\begin{tabular}{|r|r||r|r|r|r|r|r|}
\hline
\multirow{2}{*}{$n$} & \multirow{2}{*}{$|\mathcal{N}_n^{8}|$}  & \multicolumn{6}{c|}{$\varphi$} \\
\cline{3-8}
& & 12 & 14 & 16 & 18 & 20 & 22 \\
\hline
\hline
12 & 24 &  & 3 & 19 & 2 & & \\
14 & 424088 & 90 & 5963 & 106973 & 294964 & 16098 & \\
15 & 430798 &  & 19 & 4946 & 146933 & 275955 & 2945 \\
\hline
\end{tabular}%

\caption{Bipartite edge frustration of octic nut graphs.}
\label{tbl:octicFrustration}
\end{table}
\end{landscape}

\end{document}